\documentclass{article}
\usepackage[utf8]{inputenc}
\usepackage[T1]{fontenc}
\usepackage{amsmath}
\usepackage{amsthm}
\usepackage{amsfonts}
\usepackage{amssymb}
\usepackage{array}
\usepackage{dsfont}
\usepackage{hyperref}
\usepackage{enumerate}
\usepackage[all]{xy}
\usepackage{pgf,tikz}
\usepackage[caption=false]{subfig}

\def\tp{\ensuremath{\hdots}}
\def\infini{\ensuremath{\infty}}
\def\fc{\ensuremath{\mathds 1}}

\def\Cc{\ensuremath{\mathcal C}}

\def\Fc{\ensuremath{\mathcal{F}}}

\def\Nc{\ensuremath{\mathcal{N}}}

\def\E{\ensuremath{\mathbb E}}

\def\L{\ensuremath{\mathbb L}}

\def\N{\ensuremath{\mathbb N}}
\def\P{\ensuremath{\mathbb P}}

\def\R{\ensuremath{\mathbb R}}

\def\T{\ensuremath{\mathbb T}}

\def\Z{\ensuremath{\mathbb Z}}

\usepackage{pgf,tikz}
\usetikzlibrary{arrows}
\usetikzlibrary[patterns]
\definecolor{cqcqcq}{rgb}{0.7529411764705882,0.7529411764705882,0.7529411764705882}
\definecolor{ffqqqq}{rgb}{1.,0.,0.}
\definecolor{ffqqtt}{rgb}{1.,0.,0.2}
\definecolor{qqwuqq}{rgb}{0.,0.39215686274509803,0.}
\definecolor{qqqqff}{rgb}{0.,0.,1.}
\definecolor{uuuuuu}{rgb}{0.26666666666666666,0.26666666666666666,0.26666666666666666}
\definecolor{xdxdff}{rgb}{0.49019607843137253,0.49019607843137253,1.}
\definecolor{uququq}{rgb}{0.25,0.25,0.25}

\def\am{\text{argmin}}

\usepackage{cleveref}

\newtheorem{theorem}{Theorem}
\numberwithin{theorem}{section}

\newtheorem{remark}{Remark}
\numberwithin{remark}{section}

\newtheorem{lemma}{Lemma}
\numberwithin{lemma}{section}

\newtheorem{proposition}{Proposition}
\numberwithin{proposition}{section}

\newtheorem{example}{Example}
\numberwithin{example}{section}

\numberwithin{defi}{section}

\newcommand{\psh}[2]{\ensuremath{\left\langle #1,#2\right\rangle}}
\newcommand{\somm}[2]{\ensuremath{\underset{#1}{\overset{#2}{\sum}}}}
\newcommand{\devp}[2]{\ensuremath{\frac{\partial{#1}}{\partial{#2}}}}

\title{Template estimation in computational anatomy: Fréchet means in top and quotient spaces are not consistent}
\author{
Loïc Devilliers\footnote{Université C\^ote d’Azur, Inria, France, \url{loic.devilliers@inria.fr}},
Stéphanie Allassonnière\footnote{CMAP, Ecole polytechnique, CNRS, Université Paris-Saclay, 91128, Palaiseau, France},
Alain Trouvé\footnote{CMLA, ENS Cachan, CNRS, Université Paris-Saclay, 94235 Cachan, France},\\
and Xavier Pennec\footnote{Université C\^ote d’Azur, Inria, France}}
\begin{document}
\maketitle
\begin{abstract}
In this article, we study the consistency of the template estimation with the Fréchet mean in quotient spaces.
The Fréchet mean in quotient spaces is often used when the observations are deformed or transformed by a group action. We show that in most cases this estimator is actually inconsistent. We exhibit a sufficient condition for this inconsistency, which amounts to the folding of the distribution of the noisy template when it is projected to the quotient space. This condition appears to be fulfilled as soon as the support of the noise is large enough. 
To quantify this inconsistency we provide lower and upper bounds of the bias as a function of the variability (the noise level). This shows that the consistency bias cannot be neglected when the variability increases. 
\end{abstract}
\newpage
\tableofcontents
\section{Introduction}
In Kendall's shape space theory~\cite{ken1}, in computational anatomy~\cite{gre}, in statistics on signals, or in image analysis, one often aims at estimating a template. A template stands for a prototype of the data.
The data can be the shape of an organ studied in a population~\cite{dur} or an aircraft~\cite{alla3}, an electrical signal of the human body, a MR image etc. To analyse the observations, one assumes that these data follow a statistical model. One often models observations as random deformations of the template with additional noise. This deformable template model proposed in~\cite{gre} is commonly used in computational anatomy.
The concept of deformation introduces the notion of group action: the deformations we consider are elements of a group which acts on the space of observations, called here the top space. 
Since the deformations are unknown, one usually considers equivalent classes of observations under the group action. In other words, one considers the quotient space of the top space (or ambient space) by the group. In this particular setting, the template estimation is most of the time based on the minimisation of the empirical variance in the quotient space (for instance~\cite{kur,jos,sab} among many others). 
The points that minimise the empirical variance are called the empirical Fréchet mean. The Fréchet means introduced in~\cite{fre} is comprised of the elements minimising the variance. This generalises the notion of expected value in non linear spaces. Note that the existence or uniqueness of Fréchet mean is not ensured. But sufficient conditions may be given in order to reach existence and uniqueness (for instance~\cite{kar} and~\cite{ken2}).

Several group actions are used in practice: some signals can be shifted in time compared to other signals (action of translations~\cite{hit}), landmarks can be transformed rigidly~\cite{ken1}, shapes can be deformed by diffeomorphisms~\cite{dur}, etc. In this paper we restrict to transformation which leads the norm unchanged. Rotations for instance leave the norm unchanged, but it may seem restrictive. In fact, the square root trick detailed in~\cref{sec:fixedpoint}, allows to build norms  which are unchanged, for instance by reparametrization of curves with a diffeomorphism, where our work can be applied.

We raise several issues concerning the estimation of the template.
\begin{enumerate}
\item Is the Fréchet mean in the quotient space equal to the original template projected in the quotient space? In other words, is the template estimation with the Fréchet mean in quotient space consistent?
\item If there is an inconsistency, how large is the consistency bias? 
Indeed, we may expect the consistency bias to be negligible in many practicable cases.
\item If one gets only a finite sample, one can only estimate the empirical Fréchet mean. How far is the empirical Fréchet mean from the original template?
\end{enumerate}
These issues originated from an example exhibited by
Allassonnière, Amit and Trouvé~\cite{all}: they took a step function as a template and they added some noise and shifted in time this function. By repeating this process they created a data sample from this template. With this data sample, they tried to estimate the template with the empirical Fréchet mean in the quotient space. 
In this example, minimising the empirical variance did not succeed in estimating well the template when the noise added to the template increases, even with a large sample size.

One solution to ensure convergence to the template is to replace this estimation method with a Bayesian paradigm (\cite{all2,bon} or~\cite{miao}). But there is a need to have a better understanding of the failure of the template estimation with the Fréchet mean. One can studied the inconsistency of the template estimation. Bigot and Charlier~\cite{big} first studied the question of the template estimation with a finite sample in the case of translated signals or images by providing a lower bound of the consistency bias. This lower bound was unfortunately not so informative as it is converging to zero asymptotically when the dimension of the space tends to infinity.
Miolane et al.~\cite{mio,mio2} later provided a more general explanation of why the template is badly estimated for a general group action thanks to a geometric interpretation. They showed that the external curvature of the orbits is responsible for the inconsistency. This result was further quantified with Gaussian noise. In this article, we provide sufficient conditions on the noise for which inconsistency appears and we quantify the consistency bias in the general (non necessarily Gaussian) case. Moreover,  we mostly consider a vector space (possibly infinite dimensional) as the top space while  
the article of Miolane et al. is restricted to finite dimensional manifolds. 
In a preliminary unpublished version of this work~\cite{dev}, we proved the inconsistency when the transformations come from a finite group acting by translation. The current article extends these results by generalizing to any isometric action of finite and non-finite groups.

This article is organised as follows. \Cref{subsec:mintro} details the mathematical terms that we use and the generative model. In \cref{sec:bias,sec:general}, we exhibit sufficient condition that lead to an inconsistency when the template is not a fixed point under the group action. This sufficient condition can be roughly understand as follows: with a non zero probability, the projection of the random variable on the orbit of the template is different from the template itself. This condition is actually quite general. In particular, this condition it is always fulfilled with the Gaussian noise or with any noise whose support is the whole space. Moreover we quantify the consistency bias with lower and upper bounds. We restrict our study to Hilbert spaces and isometric actions. This means that the space is linear, the group acts linearly and leaves the norm (or the dot product) unchanged. \Cref{sec:bias} is dedicated to finite groups. Then we generalise our result in \cref{sec:general} to non-finite groups. To complete this study, we extend in \cref{sec:fixedpoint} the result when the template is a fixed point under the group action and when the top space is a manifold. As a result we show that the inconsistency exists for almost all noises. Although the bias can be neglected when the noise level is sufficiently small, its linear asymptotic behaviour with respect to the noise level show that it becomes unavoidable for large noises.

\section{Definitions, notations and generative model}
\label{subsec:mintro}

 We denote by $M$ the top space, which is the image/shape space, and $G$ the group acting on $M$. The action is a map: 
\begin{equation*}
    \begin{matrix}
    G\times M&\to & M\\
    (g,m)&\mapsto & g\cdot m
    \end{matrix}
\end{equation*}
satisfying the following properties: for all $g,\: g'\in G$, $m\in M$ $(gg')\cdot m=g\cdot (g'\cdot m)$ and $e_G\cdot m=m$ where $e_G$ is the neutral element of $G$. For $m\in M$ we note by $[m]$ the orbit of $m$ (or the class of $m$). This is the set of points reachable from $m$ under the group action: $[m]=\{g\cdot m,\: g\in G\}$. Note that if we take two orbits $[m]$ and $[n]$ there are two possibilities:
\begin{enumerate}
\item The orbits are equal: $[m]=[n]$ i.e. $\exists g\in G \mbox{ s.t. } n=g\cdot m$. 
\item The orbits have an empty intersection: $[m]\cap [n]=\emptyset$.
\end{enumerate}
We call quotient of $M$ by the group $G$ the set all orbits. This quotient is noted by:
\begin{equation*}
Q=M/G=\{[m],\:m\in M\}.
\end{equation*}
The orbit of an element $m\in M$ can be seen as the subset of $M$ of all elements $g\cdot m$ for $g\in G$ or as a point in the quotient space. In this article we use these two ways. We project an element $m$ of the top space $M$ into the quotient by taking $[m]$.

Now we are interested in adding a structure on the quotient from an existing structure in the top space: take $M$ a metric space, with $d_M$ its distance. Suppose that $d_M$ is invariant under the group action which means that $\forall g\in G,\: \forall a,\: b\in M\: d_M(a,b)=d_M(g\cdot a,g\cdot b)$. Then we obtain a pseudo-distance on $Q$ defined by:
\begin{equation}
d_Q([a],[b])=\underset{g\in G}{\inf} d_M(g\cdot a,b).
\label{quotientdistance}
\end{equation}
We remind that a distance on $M$ is a map $d_M:M\times M\mapsto \R^+$ such that for all $m,\: n,\: p\in M$:
\begin{enumerate}
\item $d_M(m,n)=d_M(n,m)$ (symmetry).
\item $d_M(m,n)\leq d_M(m,p)+d_M(p,n)$ (triangular inequality).
\item $d_M(m,m)=0$.
\item $d_M(m,n)=0\Longleftrightarrow m=n$.
\end{enumerate}
A pseudo-distance satisfies only the first three conditions. If we suppose that all the orbits are closed sets of $M$, then one can show that $d_Q$ is a distance. In this article, we assume that $d_Q$ is always a distance, even if a pseudo-distance would be sufficient. $d_Q([a],[b])$ can be interpreted as the distance between the shapes $a$ and $b$, once one has removed the parametrisation by the group $G$. In other words, $a$ and $b$ have been registered. In this article, except in~\cref{sec:fixedpoint}, we suppose that the the group acts isometrically on an Hilbert space, this means that the map $x\mapsto g\cdot x$ is linear, and that the norm associated to the dot product is conserved: $\|g\cdot x\|=\|x\|$. Then $d_M(a,b)=\|a-b\|$ is a particular case of invariant distance. 

We now introduce \textbf{the generative model} used in this article for $M$ a vector space. Let us take a template $t_0\in M$ to which we add a unbiased noise $\epsilon$: $X=t_0+\epsilon$. Finally we transform $X$ with a random shift $S$ of $G$. We assume that this variable $S$ is independent of $X$ and the only observed variable is:
\begin{equation}
    Y=S\cdot X=S\cdot (t_0+\epsilon), \mbox{ with } \E(\epsilon)=0,
    \label{modgen}
\end{equation}
while $S$, $X$ and $\epsilon$ are hidden variables.

Note that it is not the generative model defined by Grenander and often used in computational anatomy. Where the observed variable is rather $Y'=S\cdot t_0+\epsilon'$. But when the noise is isotropic and the action is isometric, one can show that the two models have the same law, since $S\cdot \epsilon$ and $\epsilon$ have the same probability distribution. As a consequence, the inconsistency of the template estimation with the Fréchet mean in quotient space with one model implies the inconsistency with the other model. 
Because the former model~\eqref{modgen} leads to simpler computation we consider only this model.

We can now set the inverse problem: given the observation $Y$, how to estimate the template $t_0$ in $M$? This is an ill-posed problem. Indeed for some element group $g\in G$, the template $t_0$ can be replaced by the translated $g\cdot t_0$, the shift $S$ by $Sg^{-1}$ and the noise $\epsilon$ by $g\epsilon$, which leads to the same observation $Y$. So instead of estimating the template $t_0$, we estimate its orbit $[t_0]$.
By projecting the observation $Y$ in the quotient space we obtain $[Y]$.
Although the observation $Y=S\cdot X$ and the noisy template $X$ are different random variables in the top space, their projections on the quotient space lead to the same random orbit $[Y]=[X]$. 
That is why we consider the generative model~\eqref{modgen}: the projection in the quotient space remove the transformation of the group $G$.
From now on, we use the random orbit $[X]$ in lieu of the random orbit of the observation $[Y]$. 

 The variance of the random orbit $[X]$ (sometimes called the Fréchet functional or the energy function) at the quotient point $[m]\in Q$ is the expected value of the square distance between $[m]$ and the random orbit $[X]$, namely:
\begin{equation}
     Q\ni [m]\mapsto \E(d_Q([m],[X])^2)
    \label{varY}
\end{equation}
An orbit $[m]\in Q$ which minimises this map is called a Fréchet mean of $[X]$.

If we have an \textit{i.i.d} sample of observations $Y_1,\tp, Y_n$ we can write the \textit{empirical quotient variance}:
\begin{equation}
    Q\ni [m]\mapsto \frac1{n} \somm{i=1}{n} d_Q([m],[Y_i])^2=\frac{1}{n} \somm{i=1}{n} \underset{g_i\in G}{\inf} \|m-g_i\cdot Y_i\|^2.
    \label{empv}
\end{equation}
Thanks to the equality of the quotient variables $[X]$ and $[Y]$, an element which minimises this map is an \textit{empirical Fréchet mean} of $[X]$.

In order to minimise the empirical quotient variance~\eqref{empv}, the max-max algorithm\footnote{The term max-max algorithm is used for instance in~\cite{all}, and we prefer to keep the same name, even if it is a minimisation.} alternatively minimises the function $J(m,(g_i)_i)=\frac{1}{n} \somm{i=1}{n} \| m-g_i \cdot Y_i\|^2$ over a point $m$ of the orbit $[m]$ and over the hidden transformation $(g_i)_{1\leq i\leq n}\in G^n$.

With these notations we can reformulate our questions as:
\begin{enumerate}
    \item Is the orbit of the template $[t_0]$ a minimiser of the quotient variance defined in~\eqref{varY}?
If not, the Fréchet mean in quotient space is an inconsistent estimator of $[t_0]$.
    \item In this last case, can we quantify the quotient distance between $[t_0]$ and a Fréchet mean of $[X]$?
    \item Can we quantify the distance between $[t_0]$ and an empirical Fréchet mean of a $n$-sample?
\end{enumerate}
This article shows that the answer to the first question is usually "no" in the framework of an Hilbert space $M$ on which a group $G$ acts linearly and isometrically. The only exception is~\cref{theo1} where the top space $M$ is a manifold. 
In order to prove inconsistency, an important notion in this framework is the isotropy group of a point $m$ in the top space. This is the subgroup which leaves this point unchanged:
\begin{equation*}
    \mbox{Iso}(m)=\{g\in G,\: g\cdot m=m\}.
\end{equation*}
We start in~\cref{sec:bias} with the simple example where the group is finite and the isotropy group of the template is reduced to the identity element ($Iso(t_0)=\{e_G\}$, in this case $t_0$ is called a regular point). We turn in~\cref{sec:general} to the case of a general group and an isotropy group of the template which does not cover the whole group ($Iso(t_0)\neq G$) i.e $t_0$ is not a fixed point under the group action.
To complete the analysis, we assume in~\cref{sec:fixedpoint} that the template $t_0$ is a fixed point which means that $Iso(t_0)=G$.

In~\cref{sec:bias,sec:general} we show lower and upper bounds of the consistency bias which we define as the quotient distance between the template orbit and the Fréchet mean in quotient space. These results give an answer to the second question. In~\cref{sec:general}, we show a lower bound for the case of the empirical Fréchet mean which answers to the third question. 

As we deal with different notions whose name or definition may seem similar, we use the following vocabulary: 
\begin{enumerate}
    \item The variance of the noisy template $X$ in the top space is the function $E:m\in M\mapsto \E(\|m-X\|^2)$. The unique element which minimises this function is the Fréchet mean of $X$ in the top space. With our assumptions it is the template $t_0$ itself.
\item We call variability (or noise level) of the template the value of the variance at this minimum: $\sigma^2=\E(\|t_0-X\|^2)=E(t_0)$.
    \item The variance of the random orbit $[X]$ in the quotient space is the function $F:m\mapsto \E(d_Q([m],[X])^2)$. 
Notice that we define this function from the top space and not from the quotient space.
With this definition, an orbit $[m_\star]$ is a Fréchet mean of $[X]$ if the point $m_\star$ is a global minimiser of $F$.
\end{enumerate}

In~\cref{sec:bias,sec:general}, we exhibit a sufficient condition for the inconsistency, which is:
the noisy template $X$ takes value with a non zero probability in the set of points which are strictly closer to $g\cdot t_0$ for some $g\in G$ than the template $t_0$ itself. 
This is linked to the folding of the distribution of the noisy template when it is projected to the quotient space. The points for which the distance to the template orbit in the quotient space is equal to the distance to the template in the top space are projected without being folded.
If the support of the distribution of the noisy template contains folded points (we only assume that the probability measure of $X$, noted $\P$, is a regular measure), then there is inconsistency. The support of the noisy template $X$ is defined by the set of points $x$ such that $\P(X\in B(x,r))>0$ for all $r>0$.
For different geometries of the orbit of the template, we show that this condition is fulfilled as soon as the support of the noise is large enough. 

The recent article of Cleveland et al.~\cite{cle} may seem contradictory with our current work. Indeed the consistency of the template estimation with the Fréchet mean in quotient space is proved under hypotheses which seem to satisfy our framework: the norm is unchanged under their group action (isometric action) and a noise is present in their generative model. However we believe that the noise they consider might actually not be measurable. Indeed, their top space is:
\begin{equation*}
    L^2([0,1])=\left\{f:[0,1]\to \R \mbox{ such that } f \mbox{ is measurable and } \int_0^1 f^2(t)dt<+\infini\right\}.
\end{equation*}
The noise $e$ is supposed to be in $L^2([0,1])$ such that for all $t,\:s\in [0,1]$, $\E(e(t))=0$ and $\E(e(t)e(s))=\sigma^2 \fc_{s=t}$, for $\sigma>0$. This means that $e(t)$ and $e(s)$ are chosen without correlation as soon as $s\neq t$. In this case, it is not clear for us that the resulting function $e$ is measurable, and thus that its Lebesgue integration makes sense. Thus, the existence of such a random process should be established before we can fairly compare the results of both works.

\section{Inconsistency for finite group when the template is a regular point}
\label{sec:bias}
In this Section, we consider a finite group $G$ acting isometrically and effectively on $M=\R^n$ a finite dimensional space equipped with the euclidean norm $\|\quad\|$, associated to the dot product $\psh{\quad}{\quad}$.

We say that the action is effective if $x\mapsto g\cdot x$ is the identity map if and only if $g=e_G$. Note that if the action is not effective, we can define a new effective action by simply quotienting $G$ by the subgroup of the element $g\in G$ such that $x\mapsto g\cdot x$ is the identity map. 

The template is assumed to be a regular point which means that the isotropy group of the template is reduced to the neutral element of $G$. Note that the measure of singular points (the points which are not regular) is a null set for the Lebesgue measure (see~\cref{singnull} in~\cref{subsec:prop}).

\begin{example}
\label{ex2}
The action of translation on coordinates: this action is a simplified setting for image registration, where images can be obtained by the translation of one scan to another due to different poses. More precisely, we take the vector space $M=\R^\T$ where $G=\T=(\Z/N\Z)^D$ is the finite torus in $D$-dimension. An element of $\R^\T$ is seen as a function $m:\T\to \R$, where $m(\tau)$ is the grey value at pixel $\tau$. When $D=1$, $m$ can be seen like a discretised signal with $N$ points, when $D=2$, we can see $m$ like an image with $N\times N$ pixels etc. We then define the group action of $\T$ on $\R^\T$ by: 
\begin{equation*}
\tau\in \T,\: m\in \R^\T \quad \tau \cdot m:\sigma\mapsto m(\sigma+\tau).
\end{equation*}
This group acts isometrically and effectively on $M=\R^ \T$.
\end{example}

In this setting, if $\E(\|X\|^2)<+\infini$ then the variance of $[X]$ is well defined:
\begin{equation}
F:m\in M\mapsto \E(d_Q([X],[m])^2).
\label{defJ}
\end{equation}
In this framework, $F$ is non-negative and continuous. Thanks to Cauchy-Schwarz inequality we have: 
\begin{equation*}
\underset{\|m\| \to \infini}{\lim} F(m)\geq \underset{\|m\|\to \infini}{\lim} \|m\|^2 -2\|m\| \E(\|X\|)+\E(\|X\|^2)=+\infini.
\end{equation*}
Thus for some $R>0$ we have: for all $m\in M$ if $\|m\|>R$ then $F(m)\geq F(0)+1$. The closed ball $B(0,R)$ is a compact set (because $M$ is a finite vector space) then $F$ restricted to this ball reached its minimum $m_\star$. Then for all $m\in M$, if $m\in B(0,R)$, $F(m_\star)\leq F(m)$, if $\|m\|>R$ then $F(m)\geq F(0)+1>F(0)\geq F(m_\star)$. Therefore $[m_\star]$ is a Fréchet mean of $[X]$ in the quotient $Q=M/G$. Note that this ensure the existence but not the uniqueness. 

In this Section, we show that as soon as the support of the distribution of $X$ is big enough, the orbit of the template is not a Fréchet mean of $[X]$. We provide a upper bound of the consistency bias depending on the variability of $X$ and an example of computation of this consistency bias.

\subsection{Presence of inconsistency}
\label{subsec:biascondition}

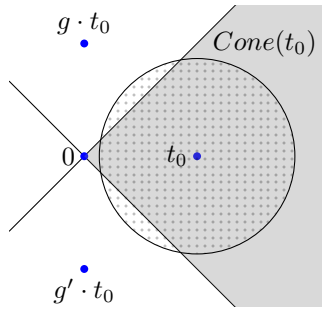
\begin{figure}[!ht]\centering\begin{tikzpicture}
\clip (-1,-2) rectangle (3.2,2);

\draw (-1,-1)--(3,3);
\draw (-1, 1)--(3,-3);

\fill [color=blue] (0,0) circle (1.5pt);
\fill [color=blue] (1.5,0) circle (1.5pt);
\fill [color=blue] (0,1.5) circle (1.5pt);
\fill [color=blue] (0,-1.5) circle (1.5pt);
\fill [color=gray, opacity=0.3] (0,0)--(3.5,3.5)--(3.5,-3.5)--(0,0);

\fill [pattern=dots, opacity=0.5] (1.5,0) circle (1.3);
\draw (1.5,0) circle (1.3);
\draw (0,0) node[left] {$0$};
\draw (1.5,0) node[left] {$t_0$};
\draw (0,1.5) node[above] {$g\cdot t_0$};
\draw (0,-1.5) node[below] {$g'\cdot t_0$};
\draw (2.4,1.8) node[below] {$Cone(t_0)$};

\end{tikzpicture}\caption[Cone of the template]{Planar representation of a part of the orbit of the template $t_0$. The lines are the hyperplanes whose points are equally distant of two distinct elements of the orbit of $t_0$, $Cone(t_0)$ represented in points is the set of points closer from $t_0$ than any other points in the orbit of $t_0$. \Cref{theo} states that if the support (the dotted disk) of the random variable $X$ is not included in this cone, then there is an inconsistency.}\label{fig:Pg}\end{figure}

The following theorem gives a sufficient condition on the random variable $X$ for an inconsistency:
\begin{theorem}
\label{theo}
Let $G$ be a finite group acting on $M=\R^n$ isometrically and effectively. Assume that the random variable $X$ is absolutely continuous with respect to the Lebesgue's measure, with $\E(\|X\|^2)<+\infini$. We assume that $t_0=\E(X)$ is a regular point. 

We define $Cone(t_0)$ as the set of points closer from $t_0$ than any other points of the orbit $[t_0]$, see \cref{fig:Pg} or~\cref{dcone} in~\cref{subsec:prop} for a formal definition. In other words, $Cone(t_0)$ is defined as the set of points already registered with $t_0$.
Suppose that:
\begin{equation}
\P\left(X\notin Cone(t_0)\right)>0,
\label{cone}
\end{equation}
then $[t_0]$ is not a Fréchet mean of $[X]$. 
\end{theorem}

The proof of \cref{theo} is based on two steps: first, differentiating the variance $F$ of $[X]$. Second, showing that the gradient at the template is not zero, therefore the template can not be a minimum of $F$. \Cref{difftheo} makes the first step. 
\begin{theorem}
\label{difftheo}
The variance $F$ of $[X]$ is differentiable at any regular points. For $m_0$ a regular point, we define $g(x,m_0)$ as the almost unique $g\in G$ minimising $\|m_0-g\cdot x\|$ (in other words, $g(x, m_0)\cdot x\in Cone(m_0)$). This allows us to compute the gradient of $F$ at $m_0$:
\begin{equation}
\nabla F(m_0)=2(m_0-\E( g(X,m_0)\cdot X)).
\label{nablaJ}
\end{equation}

\end{theorem}

\begin{figure}[!ht]
\centering
\subfloat[Graphic representation of the template $t_0=\E(X)$ mean of points of the support of $X$. ]{\begin{tikzpicture}
\clip (-1,-2) rectangle (3.2,2);

\draw (-1,-1)--(3,3);
\draw (-1, 1)--(3,-3);

\fill [color=blue] (0,0) circle (1.5pt);
\fill [color=blue] (1.5,0) circle (1.5pt);
\fill [color=blue] (0,1.5) circle (1.5pt);
\fill [color=blue] (0,-1.5) circle (1.5pt);
\fill [color=gray, opacity=0.3] (0,0)--(3.5,3.5)--(3.5,-3.5)--(0,0);

\fill [pattern=dots, opacity=0.5] (1.5,0) circle (1.3);
\draw (1.5,0) circle (1.3);
\draw (0,0) node[left] {$0$};
\draw (1.5,0) node[left] {$t_0$};
\draw (0,1.5) node[above] {$g\cdot t_0$};
\draw (0,-1.5) node[below] {$g'\cdot t_0$};
\draw (2.4,1.8) node[below] {$Cone(t_0)$};

\end{tikzpicture}\label{fig:DY}}
\quad 
\subfloat[Graphic representation of $Z=\E(g(X,t_0)\cdot X)$. The points $X$ which were outside $Cone(t_0)$ are now in $Cone(t_0)$ (thanks to $g(X,t_0)$). This part, in grid-line, represents the points which have been folded.]{\begin{tikzpicture}
\clip (-1,-2) rectangle (3.2,2);

\draw (-1,-1)--(3,3);
\draw (-1, 1)--(3,-3);

\draw (0,0) node[left] {$0$};
\draw (1.5,0) node[left] {$t_0$};
\draw (1.7,0.3) node[right] {$Z$};
\draw (0,1.5) node[above] {$g\cdot t_0$};
\draw (0,-1.5) node[below] {$g'\cdot t_0$};
\draw (2.4,1.8) node[below] {$Cone(t_0)$};

\fill [color=blue] (0,0) circle (1.5pt);
\fill [color=blue] (1.5,0) circle (1.5pt);
\fill [color=blue] (1.7,0.3) circle (1.5pt);
\fill [color=blue] (0,1.5) circle (1.5pt);
\fill [color=blue] (0,-1.5) circle (1.5pt);
\fill [color=gray, opacity=0.3] (0,0)--(3.5,3.5)--(3.5,-3.5)--(0,0);

\fill[pattern=dots,opacity=0.5] (0.222,-0.222)--(1.28,-1.28)--(1.28,1.28)--(0.222,0.222)--(0.222,0.222);
\filldraw[draw=black,pattern=dots,opacity=0.5,shift={(1.5,0)}] plot[domain=-1.7410082520803716:1.7410082520803711,variable=\t]({1*1.2987686476043376*cos(\t r)+0*1.2987686476043376*sin(\t r)},{0*1.2987686476043376*cos(\t r)+1*1.2987686476043376*sin(\t r)});

\filldraw [draw=black,pattern=north east lines, shift={(0,1.5)}] plot[domain=4.882600905670164:6.112973381894112,variable=\t]({1*1.2987686476043376*cos(\t r)+0*1.2987686476043376*sin(\t r)},{0*1.2987686476043376*cos(\t r)+1*1.2987686476043376*sin(\t r)});

\filldraw[draw=black,pattern=north east lines,shift={(0,-1.5)}] plot[domain=0.17021192528547438:1.4005844015094222,variable=\t]({1*1.2987686476043376*cos(\t r)+0*1.2987686476043376*sin(\t r)},{0*1.2987686476043376*cos(\t r)+1*1.2987686476043376*sin(\t r)});

\filldraw[draw=black,pattern=dots,opacity=0.5,shift={(1.5,0)}] plot[domain=2.971380728304319:3.3118045788752672,variable=\t]({1*1.2987686476043376*cos(\t r)+0*1.2987686476043376*sin(\t r)},{0*1.2987686476043376*cos(\t r)+1*1.2987686476043376*sin(\t r)});
\end{tikzpicture}\label{fig:DZ}}
\caption[Gradient of the variance at the template is not zero.]{$Z$ is the mean of points in $Cone(t_0)$ where $Cone(t_0)$ is the set of points closer from $t_0$ than $g\cdot t_0$ for $g\in G\setminus{e_G}$. Therefore it seems that $Z$ is higher that $t_0$, therefore $\nabla F(t_0)=2(t_0-Z)\neq 0$.}
\label{fig:DYZ}
\end{figure}
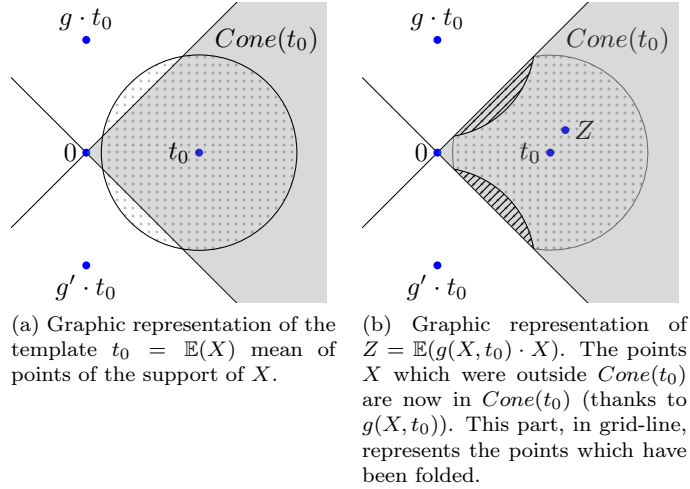


This Theorem is proved in~\cref{subsec:prop}. Then we show that the gradient of $F$ at $t_0$ is not zero. To ensure that $F$ is differentiable at $t_0$ we suppose in the assumptions of~\cref{theo} that $t_0=\E(X)$ is a regular point. Thanks to~\cref{difftheo} we have:
\begin{equation*}
\nabla F(t_0)=2(t_0-\E( g(X,t_0)\cdot X)).
\label{difft0}
\end{equation*}

Therefore $\nabla F(t_0)/2$ is the difference between two terms, which are represented on~\cref{fig:DYZ}: on~\cref{fig:DY} there is a mass under the two hyperplanes outside $Cone(t_0)$, so this mass is nearer from $gt_0$ for some $g\in G$ than from $t_0$. In the following expression $
Z=\E( g(X,t_0)\cdot X)$, for $X\notin Cone(t_0)$, $g(X,t_0)X\in Cone(t_0)$ such points are represented in grid-line on~\cref{fig:DYZ}. This suggests that the point $Z=\E( g(X,t_0)\cdot X)$ which is the mean of points in $Cone(t_0)$  is further away from $0$ than $t_0$. Then $\nabla F(t_0)/2=t_0-Z$ should be not zero, and $t_0=\E(X)$ is not a critical point of the variance of $[X]$. As a conclusion $[t_0]$ is not a Fréchet mean of $[X]$. This is turned into a rigorous proof in~\cref{subsec:prooftheo}.

In the proof of~\cref{theo}, we took $M$ an Euclidean space and we work with the Lebesgue's measure in order to have $\P(X\in H)=0$ for every hyperplane $H$. Therefore the proof of~\cref{theo} can be extended immediately to any Hilbert space $M$, if we make now the assumption that $\P(X\in H)=0$ for every hyperplane $H$, as long as we keep a finite group acting isometrically and effectively on $M$.

\Cref{fig:DYZ} illustrates the condition of \cref{theo}: if there is no mass beyond the hyperplanes, then the two terms in $\nabla F(t_0)$ are equal (because almost surely $g(X,t_0)\cdot X=X$). Therefore in this case we have $\nabla F(t_0)=0$. This do not prove necessarily that there is no inconsistency, just that the template $t_0$ is a critical point of $F$. Moreover this figure can give us an intuition on what the consistency bias (the distance between $[t_0]$ and the set of all Fréchet mean in the quotient space) depends: for $t_0$ a fixed regular point, when the variability of $X$ (defined by $\E(\|X-t_0\|^2)$) increases the mass beyond the hyperplanes on \cref{fig:DYZ} also increases, the distance between $\E( g(X,t_0)\cdot X)$ and $t_0$ (i.e. the norm of $\nabla F(t_0)$) augments. Therefore $q$ the Fréchet mean should be further from $t_0$, (because at this point one should have $\nabla F(q)=0$ or $q$ is a singular point). Therefore the consistency bias appears to increase with the variability of $X$. By establishing a lower and upper bound of the consistency bias and by computing the consistency bias in a very simple case, \cref{subsec:quanti,subsec:Exemple,subsec:lowerbound,subsec:up} investigate how far this hypothesis is true.

We can also wonder if the converse of \cref{theo} is true: if the support is included in $Cone(t_0)$, is there consistency? We do not have a general answer to that. In the simple example \cref{subsec:Exemple} it happens that condition~\eqref{cone} is necessary and sufficient. More generally the following proposition provides a partial converse:

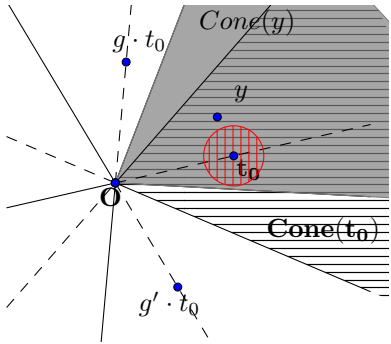
\begin{figure}[!ht]	
\centering
\begin{tikzpicture}[line cap=round,line join=round,>=triangle 45,x=0.8cm,y=0.8cm]
\clip(-2.84,-0.9) rectangle (3.5,4.68);
\fill[fill=black,pattern=horizontal lines,pattern color=black,thick] (-1.0365112624723687,1.731850250241854) -- (2.730851205393955,6.0505340548691) -- (5.808224746961198,-1.1876073536357676) -- cycle;
\fill[color=gray,fill=gray,fill opacity=0.7] (-1.0365112624723687,1.731850250241854) -- (1.0265745945705498,7.078599404177075) -- (6.393949964931377,1.3294950741535771) -- cycle;
\draw (0.82,2.28) node[anchor=north west] {$\mathbf{t_0}$};
\draw (-1.24,4.42) node[anchor=north west] {$g\cdot t_0$};
\draw [domain=-2.8399999999999985:-1.0365112624723687] plot(\x,{(-0.--1.7318502502418531*\x)/-1.0365112624723687});
\draw [domain=-2.8399999999999985:-1.0365112624723687] plot(\x,{(-2.394437791922347-2.0103423180744766*\x)/-0.17940081007001307});
\draw [dash pattern=on 4pt off 4pt,domain=-1.0365112624723687:3.1999999999999984] plot(\x,{(--2.3944377919223463--2.0103423180744757*\x)/0.17940081007001296});
\draw [dash pattern=on 4pt off 4pt,domain=-1.0365112624723687:3.1999999999999984] plot(\x,{(-0.-1.731850250241854*\x)/1.0365112624723687});
\draw (-0.8,0.08) node[anchor=north west] {$g'\cdot t_0$};
\draw [domain=-1.0365112624723687:3.1999999999999984] plot(\x,{(--7.432605742278237--2.9178663239074543*\x)/2.5453727506426747});
\draw [dash pattern=on 4pt off 4pt,domain=-2.8399999999999985:-1.0365112624723687] plot(\x,{(-5.435189126651498-2.133727506426735*\x)/-1.8613367609254503});
\draw [domain=-1.0365112624723687:3.1999999999999984] plot(\x,{(--2.3944377919223467-0.7918502502418541*\x)/1.8565112624723685});
\draw [dash pattern=on 4pt off 4pt,domain=-2.8399999999999985:-1.0365112624723687] plot(\x,{(-2.8333018713386857--0.9369843741185866*\x)/-2.1967815793206853});
\draw [dash pattern=on 4pt off 4pt,domain=-1.0365112624723687:3.1999999999999984] plot(\x,{(--3.8742817312776046--0.450609631350424*\x)/1.9673870607049007});
\draw [domain=-2.8399999999999985:-1.0365112624723687] plot(\x,{(-3.874281731277604-0.450609631350424*\x)/-1.9673870607049007});
\draw (1.32,1.32) node[anchor=north west] {$\mathbf{Cone(t_0)}$};
\draw (-1.46,1.8) node[anchor=north west] {$\mathbf{O}$};
\draw [domain=-1.0365112624723687:3.1999999999999984] plot(\x,{(--3.3772179327520977-0.1091315888432276*\x)/2.015378669831401});
\draw [domain=-1.0365112624723687:3.1999999999999984] plot(\x,{(--6.1583861230983405--3.6124662060075994*\x)/1.3938989326204823});
\draw (0.78,3.5) node[anchor=north west] {$y$};
\draw [color=gray] (-1.0365112624723687,1.731850250241854)-- (1.0265745945705498,7.078599404177075);
\draw [color=gray] (1.0265745945705498,7.078599404177075)-- (6.393949964931377,1.3294950741535771);
\draw [color=gray] (6.393949964931377,1.3294950741535771)-- (-1.0365112624723687,1.731850250241854);
\draw (0.2,4.78) node[anchor=north west] {$Cone(y)$};
\draw [color=red,fill=red,pattern=vertical lines,pattern color=red] (0.930875798232532,2.182459881592278) circle (0.4cm);
\begin{scriptsize}
\draw [fill=blue] (0.,0.) circle (1.5pt);
\draw [fill=blue] (0.930875798232532,2.182459881592278) circle (1.5pt);
\draw [fill=blue] (-0.8571104524023557,3.7421925683163297) circle (1.5pt);
\draw [fill=blue] (-1.0365112624723687,1.731850250241854) circle (1.5pt);
\draw [fill=blue] (0.6581102700033057,2.828170800156262) circle (1.5pt);
\draw [fill=blue] (-1.0365112624723687,1.731850250241854) circle (1.5pt);
\draw [fill=blue] (-1.0365112624723687,1.731850250241854) circle (1.5pt);
\draw [fill=blue] (0.35738767014811357,5.344316456249453) circle (2.5pt);
\draw [fill=blue] (-1.0365112624723687,1.731850250241854) circle (1.5pt);
\end{scriptsize}
\end{tikzpicture}	
\caption[Variation of $y\mapsto Cone(y)$.]{$y\mapsto Cone(y)$ is continuous. When the support of the $X$ is bounded and included in the interior of $Cone(t_0)$ the hatched cone. For $y$ sufficiently close to the template $t_0$, the support of the $X$ (the ball in red) is still included 	in $Cone(y)$ (in grey), then $F(y)=(\E(\|X-y\|^2$). Therefore in this case, $[t_0]$ is at least a Karcher mean of $[X]$.}	
\label{fig:conety}
\end{figure}

\begin{proposition}
\label{kar}
If the support of $X$ is a compact set included in the interior of $Cone(t_0)$, then the orbit of the template $[t_0]$ is at least a Karcher mean of $[X]$ (a Karcher mean is a local minimum of the variance).
\end{proposition}

\begin{proof}
If the support of $X$ is a compact set included in the interior of $Cone(t_0)$ then we know that $X$-almost surely: $d_Q([X],[t_0])=\|X-t_0\|$. Thus the variance at $t_0$ in the quotient space is equal to the variance at $t_0$ in the top space. Now by continuity of the distance map (see \cref{fig:conety}) for $y$ in a small neighbourhood of $t_0$, the support of $X$ is still included in the interior of $Cone(y)$. We still have $d_Q([X],[y])=\|X-y\|$ $X$-almost surely. In other words, locally around $t_0$, the variance in the quotient space is equal to the variance in the top space. Moreover we know that $t_0=\E(X)$ is the only global minimiser of the variance of $X$: $m\mapsto E(\|m-X\|^2)=E(m)$. Therefore $t_0$ is a local minimum of $F$ the variance in the quotient space (since the two variances are locally equal). Therefore $[t_0]$ is at least a Karcher mean of $[X]$ in this case. \qquad \end{proof}

\subsection{Upper bound of the consistency bias}
\label{subsec:quanti}
In this Subsection we show an explicit upper bound of the consistency bias.
\begin{theorem}
\label{alla}
When $G$ is a finite group acting isometrically on $M=\R^n$, we denote $|G|$ the cardinal of the group $G$. If $X$ is Gaussian vector: $X\sim \Nc(t_0, s^2 Id_{\R^n})$, and $m_\star\in \am\: F$, then we have the upper bound of the consistency bias:
\begin{equation}
d_Q([t_0],[m_\star])\leq s \sqrt{8\log (|G|)}.
\label{upperbound}
\end{equation}
\end{theorem}
The proof is postponed in \cref{allaproof}. 
When $X\sim \Nc(t_0,s^2 Id_n)$ the variability of $X$ is $\sigma^2=\E(||X-t_0||^2)=ns^2$ and we can write the upper bound of the bias: $d_Q([t_0],[m_\star])\leq \frac{\sigma}{\sqrt n} \sqrt{8\log |G|}$. This Theorem shows that the consistency bias is low when the variability of $X$ is small, which tends to confirm our hypothesis in \cref{subsec:biascondition}. 
It is important to notice that this upper bound explodes when the cardinal of the group tends to infinity.

\subsection{Study of the consistency bias in a simple example}
\label{subsec:Exemple}
In this Subsection, we take a particular case of \cref{ex2}: the action of translation with $\T=\Z/2\Z$. We identify $\R^\T$ with $\R^2$ and we note by $(u,v)^T$ an element of $\R^\T$. In this setting, one can completely describe the action of $\T$ on $\R^\T$: $0\cdot(u,v)^T=(u,v)^T$ and $1\cdot(u,v)^T=(v,u)^T$. The set of singularities is the line $L=\{(u,u)^T,\: u\in \R\}$. We note $HP_A=\{(u,v)^T,\: v> u\}$ the half-plane above $L$ and $HP_B$ the half-plane below $L$. This simple example will allow us to provide necessary and sufficient condition for an inconsistency at regular and singular points. Moreover we can compute exactly the consistency bias, and exhibit which parameters govern the bias. We can then find an equivalent of the consistency bias when the noise tends to zero or infinity. More precisely, we have the following theorem proved in \cref{p2p}:
\begin{proposition} 
\label{2p}
Let $X$ be a random variable such that $\E(\|X\|^2)<+\infini$ and $t_0=\E(X)$.
\begin{enumerate}
\item \label{item1}
If $t_0\in L$, there is no inconsistency if and only if the support of $X$ is included in the line $L=\{(u,u),\: u\in \R\}$.
If $t_0\in HP_A$ (respectively in $HP_B$), there is no inconsistency if and only if the support of $X$ is included in $HP_A\cup L$ (respectively in $HP_B\cup L$).

\item 
If $X$ is Gaussian: $X\sim \Nc(t_0,s^2 Id_2)$, then the Fréchet mean of $[X]$ exists and is unique. 
This Fréchet mean $[m_\star]$ is on the line passing through $\E(X)$ and perpendicular to $L$ and the consistency bias $\tilde \rho=d_Q([t_0],[m_\star])$ is the function of $s$ and $d=\mbox{dist}(t_0,L)$ given by:
\begin{equation}
\tilde \rho(d,s) = s\frac{2}{\pi} \int_{\frac{d}{s}}^{+\infini} r^2\exp\left(-\frac{r^2}{2}\right)g\left(\frac{d}{rs}\right)dr,
\label{item2}
\end{equation}
where $g$ is a non-negative function on $[0,1]$ defined by $g(x)=\sin(\arccos(x))-x\arccos(x)$.
\begin{enumerate}
\item \label{itema} If $d> 0$ then $s\mapsto \tilde \rho(d,s)$ has an asymptotic linear expansion:
\begin{equation}
\tilde \rho(d,s) \underset{s\to \infini}{\sim} s\frac2\pi \int_0^{+\infini}r^2 \exp\left(-\frac{r^2}2\right)dr.
\end{equation}
\item \label{itemb}If $d>0$, then $\tilde\rho(d,s)=o(s^k)$ when $s\to 0$, for all $k\in \N$.
\item \label{itemc}$s\mapsto \tilde \rho(0,s)$ is linear with respect to $s$ (for $d=0$ the template is a fixed point).
\end{enumerate}
\end{enumerate}
\end{proposition}

\begin{remark} 
Here, contrarily to the case of the action of rotation in~\cite{mio2}, it is not the ratio $\|\E (X)\|$ over the noise which matters to estimate the consistency bias. Rather the ratio $\mbox{dist}(\E(X),L)$ over the noise. However in both cases we measure the distance between the signal and the singularities which was $\{0\}$ in \cite{mio2} for the action of rotations, $L$ in this case.
\end{remark}

\section{Inconsistency for any group when the template is not a fixed point}
\label{sec:general}
In \cref{sec:bias} we exhibited sufficient condition to have an inconsistency, restricted to the case of finite group acting on an Euclidean space. We now generalize this analysis to Hilbert spaces of any dimension included infinite. Let $M$ be such an Hilbert space with its dot product noted by $\psh{\quad}{\quad}$ and its associated norm $\|\quad\|$. In this section, we do not anymore suppose that the group $G$ is finite. In the following, we prove that there is an inconsistency in a large number of situations, and we quantify the consistency bias with lower and upper bounds.

\begin{example}
The action of continuous translation: We take $G=(\R/\Z)^D$ acting on $M=\L^2((\R/\Z)^D,\R)$ with:
\begin{equation*}
\forall \tau\in G\quad \forall f\in M \quad
(\tau\cdot f):t \mapsto f(t+\tau)
\end{equation*}
This isometric action is the continuous version of the \cref{ex2}: the elements of $M$ are now continuous images in dimension $D$.
\end{example}


\subsection{Presence of an inconsistency}
\label{subsec:gbias}
We state here a generalization of \cref{theo}:
\begin{theorem}
\label{theogeneral}
Let $G$ be a group acting isometrically on $M$ an Hilbert space, and $X$ a random variable in $M$,  $\E(\|X\|^2)<+\infini$ and $\E(X)=t_0\neq 0$. If:
\begin{equation}
    \P\left( d_Q([t_0],[X])<\|t_0-X\| \right)>0,
    \label{derive2}
\end{equation}
or equivalently:
\begin{equation}
\P\left( \underset{g\in G}{\sup} \psh{g\cdot X}{t_0}> \psh{X}{t_0}\right)>0.
\label{condition}
\end{equation}
Then $[t_0]$ is not a Fréchet mean of $[X]$ in $Q=M/G$.
\end{theorem}

The condition of this Theorem is the same condition of \cref{theo}: the support of the law of $X$ contains points closer from $gt_0$ for some $g$ than $t_0$. Thus the condition~\eqref{condition} is equivalent to $\E(d_Q([X],[t_0])^2)<\E(\|X-t_0\|^2)$. In other words, the variance in the quotient space at $t_0$ is strictly smaller than the variance in the top space at $t_0$.

\begin{proof}
First the two conditions are equivalent by definition of the quotient distance and by expansion of the square norm of $\|t_0-X\|$ and of $\|t_0-gX\|$ for $g\in G$. 

As above, we define the variance of $[X]$ by:
\begin{equation*}
F(m)=\E\left( \underset{g\in G}{\inf} \|g\cdot X-m\|^2\right).
\end{equation*}
In order to prove this Theorem, we find a point $m$ such that $F(m)<F(t_0)$, which directly implies that $[t_0]$ is not be a Fréchet mean of $[X]$. 

In the proof of \cref{theo}, we showed that under condition~\eqref{cone} we had $\psh{\nabla F(t_0)}{t_0}<0$. This leads us to study $F$ restricted to $\R^+t_0$: we define for $a\in \R^+$ $f(a)=F(at_0)
=\E( \inf_{g\in G} \|g\cdot X-a\|^2)$.
Thanks to the isometric action we can expand $f(a)$ by:
\begin{equation}
f(a)=a^2 \|t_0\|^2-2a \E\left( \underset{g\in G}{\sup} \psh{g\cdot X}{t_0}\right)+\E(\|X\|^2), 
\label{ef}
\end{equation}
and explicit the unique element of $\R^+$ which minimises $f$:
\begin{equation}
a_\star=\cfrac{\E\left( \underset{g\in G}{\sup} \psh{g\cdot X}{t_0}\right)}{\|t_0\|^2}.
\label{astar}
\end{equation}
For all $x\in M$, we have $\underset{g\in G}{\sup} \psh{g\cdot x}{t_0}\geq \psh{x}{t_0}$ and thanks to condition~\eqref{condition} we get:
\begin{equation}
\E( \underset{g\in G}{\sup} \psh{g\cdot X}{t_0})> \E(\psh{X}{t_0})=\psh{\E(X)}{t_0}=\|t_0\|^2,
\label{derive}
\end{equation}
which implies $a_\star>1$. Then $F(a_\star t_0)<F(t_0)$. \qquad
\end{proof}

Note that $\|t_0\|^2(a_\star-1)=\E\left( \sup_{g\in G} \psh{g\cdot X}{t_0}\right)- \E(\psh{X}{t_0})$ (which is positive) is exactly $-\psh{\nabla F(t_0)}{t_0}/2$ in the case of finite group, see Equation~\eqref{nablaJx}. Here we find the same expression without having to differentiate the variance $F$, which may be not possible in the current setting.

\subsection{Analysis of the condition in \cref{theogeneral}}
We now look for general cases when we are sure that Equation~\eqref{condition} holds which implies the presence of inconsistency. We saw in \cref{sec:bias} that when the group was finite, it is possible to have no inconsistency only if the support of the law is included in a cone delimited by some hyperplanes. The hyperplanes were defined as the set of points equally distant of the template $t_0$ and $g\cdot t_0$ for $g\in G$. Therefore if the cardinal of the group becomes more and more important, one could think that in order to have no inconsistency the space where $X$ should takes value becomes smaller and smaller. At the limit it leaves only at most an hyperplane. In the following, we formalise this idea to make it rigorous. We show that the cases where \cref{theogeneral} cannot be applied are not generic cases.

First we can notice that it is not possible to have the condition~\eqref{condition} if $t_0$ is a fixed point under the action of $G$. Indeed in this case $\psh{g\cdot X}{t_0}=\psh{X}{g^{-1}t_0}=\psh{X}{t_0}$). So from now, we suppose that $t_0$ is not a fixed point. Now let us see some settings when we have the condition~\eqref{derive2} and thus condition~\eqref{condition}. 

\begin{proposition}
\label{denseorbit}
Let $G$ be a group acting isometrically on an Hilbert space $M$, and $X$ a random variable in $M$, with $\E(\|X\|^2)<+\infini$ and $\E(X)=t_0 \neq 0$. If:
\begin{enumerate} 
\item $[t_0]\setminus \{t_0\}$ is a dense set in $[t_0]$.
\item There exists $\eta>0$ such that the support of $X$ contains a ball $B(t_0,\eta)$.
\end{enumerate}
Then condition~\eqref{condition} holds, and the estimator is inconsistent according to \cref{theogeneral}.
\end{proposition}

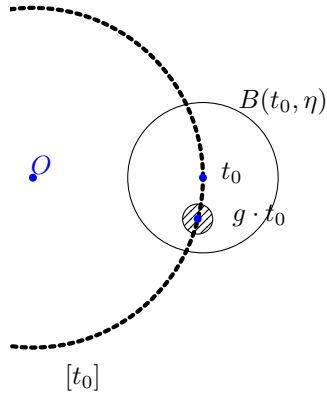
\begin{figure}[!ht]\centering\begin{tikzpicture}[line cap=round,line join=round,>=triangle 45,x=1.0cm,y=1.0cm]
\clip(1.7,-3.07) rectangle (6.13,2.47);
\draw [line width=1.6pt,dotted] (2,0) circle (2.26cm);
\draw(4.26,0) circle (1cm);
\draw (2.3,-2.36) node[anchor=north west] {$[t_0]$};
\draw (4.38,0.32) node[anchor=north west] {$t_0$};
\draw (4.61,1.32) node[anchor=north west] {$B(t_0,\eta)$};
\draw (4.53,-0.24) node[anchor=north west] {$g\cdot t_0$};
\draw [fill=black,pattern=north east lines] (4.19,-0.55) circle (0.2cm);
\fill [color=blue] (2,0) circle (1.5pt);
\draw[color=blue] (2.11,0.17) node {$O$};
\fill [color=blue] (4.26,0) circle (1.5pt);
\fill [color=blue] (4.19,-0.55) circle (1.5pt);
\end{tikzpicture}\caption[Inconsistency for the template which is an accumulation point.]{The smallest disk is included in the support of $X$ and the points in that disk is closer from $g\cdot t_0$ than from $t_0$. According to \cref{theogeneral} there is an inconsistency.}\label{fig:sphere2}\end{figure}

\begin{proof}
 By density, one takes $g\cdot t_0\in B(t_0,\eta)\setminus \{t_0\}$ for some $g\in G$, now if we take $r<\min (\|g\cdot t_0-t_0\|/2,\eta-\|g\cdot t_0-t_0\|)$ then $B(g\cdot t_0,r)\subset B(t_0,\epsilon)$. Therefore by the assumption we made on the support one has $\P(X\in B(g\cdot t_0,r))>0$. For $y\in B(g\cdot t_0,r)$ we have that $\|gt_0-y\|<\|t_0-y\|$ (see \cref{fig:sphere2}). Then we have: $\P\left( d_Q([X],[t_0])<\|X-t_0\|\right)\geq \P(X\in B(g\cdot t_0,r))>0$. Then we verify condition~\eqref{condition}, and we can apply \cref{theogeneral}.
\qquad
\end{proof}

\Cref{denseorbit} proves that there is a large number of cases where we can ensure the presence of an inconsistency. For instance when $M$ is a finite dimensional vector space and the random variable $X$ has a continuous positive density (for the Lebesgue's measure) at $t_0$, condition~2 of \Cref{denseorbit} is fulfilled. Unfortunately this proposition do not cover the case where there is no mass at the expected value $t_0=\E(X)$. This situation could appear if $X$ has two modes for instance. The following proposition deals with this situation:
\begin{proposition}
\label{propcurv}
Let $G$ be a group acting isometrically on $M$. Let $X$ be a random variable in $M$, such that $\E(\|X\|^2)<+\infini$ and $\E(X)=t_0\neq 0$. If:
\begin{enumerate}
\item $\exists \varphi \mbox{ s.t. } 
\varphi:(-a,a)\to  [t_0] \mbox{ is }\Cc^1 \mbox{ with } \varphi(0)=t_0, \varphi'(0)=v\neq 0.$
\item The support of $X$ is not included in the hyperplane $v^\perp$: $\P(X\notin v^\perp)>0$.
\end{enumerate}
Then condition~\eqref{condition} is fulfilled, which leads to an inconsistency thanks to \Cref{theogeneral}.
\end{proposition}
\begin{proof}
Thanks to the isometric action: $\psh{t_0}v=0$. We choose $y\notin v^\perp$ in the support of $X$ and make a Taylor expansion of the following square distance (see also \Cref{fig:sphere}) at $0$:
\begin{equation*}
\|\varphi(x)-y\|^2=\|t_0+xv+o(x)-y\|^2=\|t_0-y\|^2-2x\psh{y}{v}+o(x).
\end{equation*}
Then: $\exists x_\star\in (-a,a) \mbox{ s.t. } \|x_\star\|<a$, $x\psh{y}{v}>0$ and $\|\varphi(x_\star)-y\|<\|t_0-y\|$. For some $g\in G$, $\varphi(x_\star)=g\cdot t_0$. By continuity of the norm we have:
\begin{equation*}
\exists r>0\mbox{ s.t. } \forall z\in B(y,r)\quad \|g\cdot t_0-z\|<\|t_0-z\|.
\end{equation*}
Then $\P(\|g\cdot t_0-X\|<\|t_0-X\|)\geq \P(X\in B(y,r))>0$. \Cref{theogeneral} applies.
\end{proof}

\Cref{propcurv} was a sufficient condition on inconsistency in the case of an orbit which contains a curve. This brings us to extend this result for orbits which are manifolds:
\begin{proposition}
\label{biasmanifold}
Let $G$ be a group acting isometrically on an Hilbert space $M$, $X$ a random variable in $M$, with $\E(\|X\|^2)<+\infini$. Assume $X=t_0+\sigma \epsilon$, where $t_0\neq 0$ and $\E(\epsilon)=0$, and $\E(\|\epsilon\|)=1$. We suppose that $[t_0]$ is a sub-manifold of $M$ and write $T_{t_0}[t_0]$ the linear tangent space of $[t_0]$ at $t_0$. If:
\begin{equation}
\P(X\notin T_{t_0}[t_0]^{\perp})>0,
\label{perpx}
\end{equation}
which is equivalent to:
\begin{equation}
\P(\epsilon\notin T_{t_0}[t_0]^\perp)>0,
\label{perpe}
\end{equation}
then there is an inconsistency.
\end{proposition}

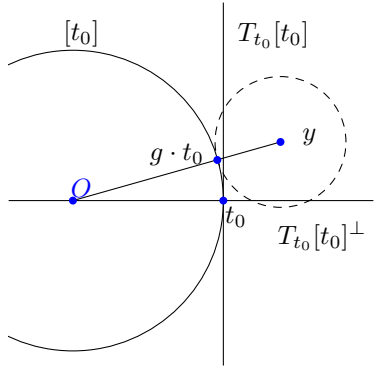
\begin{figure}[!ht]\centering\begin{tikzpicture}[line cap=round,line join=round,>=triangle 45,x=1.0cm,y=1.0cm]
\clip(-0.85,-2.19) rectangle (3.99,2.63);
\draw (2,-2.19) -- (2,2.63);
\draw(0,0) circle (2cm);
\draw (1.9,0.04) node[anchor=north west] {$ t_0 $};
\draw (-0.25,2.53) node[anchor=north west] {$[t_0]$};
\draw [domain=-0.85:3.99] plot(\x,{(-0-0*\x)/2});
\draw (2.07,2.51) node[anchor=north west] {$T_{t_0}[t_0]$};
\draw (2.6,-0.15) node[anchor=north west] {$T_{t_0}[t_0]^\perp$};
\draw (2.93,1.08) node[anchor=north west] {$y$};
\draw (0,0)-- (2.76,0.78);
\draw [dash pattern=on 3pt off 3pt] (2.76,0.78) circle (0.87cm);
\draw (0.9,0.88) node[anchor=north west] {$g\cdot t_0$};
\fill [color=blue] (0,0) circle (1.5pt);
\draw[color=blue] (0.11,0.17) node {$O$};
\fill [color=blue] (2,0) circle (1.5pt);
\fill [color=blue] (2.76,0.78) circle (1.5pt);
\fill [color=blue] (1.92,0.54) circle (1.5pt);
\end{tikzpicture}\caption[Inconsistency when the noise is not included in the Normal space at the template.]{$y\notin T_{t_0}[t_0]^\perp$ therefore $y$ is closer from $g\cdot t_0$ for some $g\in G$ than $t_0$ itself. In conclusion, if $y$ is in the support of $X$, there is an inconsistency.}\label{fig:sphere}\end{figure}

\begin{proof}
First $t_0\perp T_{t_0}[t_0]$ (because the action is isometric), $T_{t_0}[t_0]^\perp= t_0+T_{t_0}[t_0]^\perp$, then the event $\{X\in T_{t_0}[t_0]^\perp\}$ is equal to $\{\epsilon\in T_{t_0}[t_0]^\perp\}$. This proves that equations~\eqref{perpx} and~\eqref{perpe} are equivalent.
Thanks to assumption~\eqref{perpx}, we can choose $y$ in the support of $X$ such that $y\notin T_{t_0}[t_0]^\perp$. Let us take $v\in T_{t_0}[t_0]$ such that $\psh{y}{v}\neq 0$ and choose $\varphi$ a $\Cc^1$ curve in $[t_0]$, such that $\varphi(0)=t_0$ and $\varphi'(0)=v$. Applying \cref{propcurv} we get the inconsistency. \qquad
\end{proof}

Note that Condition~\eqref{perpx} is very weak, because $T_{t_0}[t_0]$ is a strict linear subspace of $M$.

\subsection{Lower bound of the consistency bias}\label{subsec:lowerbound}
Under the assumption of \Cref{theogeneral}, we have an element $a_\star t_0$ such that $F(a_\star t_0)<F(t_0)$ where $F$ is the variance of $[X]$. From this element, we deduce lower bounds of the consistency bias:
\begin{theorem}  
\label{lb}
Let $\delta$ be the unique positive solution of the following equation:
\begin{equation}
\delta^2+2\delta \left(\|t_0\|+\E\|X\|\right)-\|t_0\|^2(a_\star-1)^2=0.
\label{lowerboundb}
\end{equation}
Let $\delta_\star$ be the unique positive solution of the following equation:
\begin{equation}
\delta^2+2\delta \|t_0\|\left(1+\sqrt{1+\sigma^2/\|t_0\|^2}\right)-\|t_0\|^2(a_\star-1)^2=0,
\label{lowerbound}
\end{equation}
where $\sigma^2=\E(\|X-t_0\|^2)$ is the variability of $X$. Then $\delta$ and $\delta_\star$ are two lower bounds of the consistency bias.
\end{theorem}
\begin{proof}
In order to prove this Theorem, we exhibit a ball around $t_0$ such that the points on this ball have a variance bigger than the variance at the point $a_\star t_0$, where $a_\star$ was defined in Equation~\eqref{astar}: 
thanks to the expansion of the function $f$ we did in~\eqref{ef} we get :
\begin{equation}
F(t_0)-F(a_\star t_0)= \|t_0\|^2(a_\star-1)^2>0,
\label{delta}
\end{equation}
Moreover we can show (exactly like equation~\eqref{variation}) that for all $x\in M$:
\begin{align}
\left| F(t_0)-F(x)\right|&\leq 
\E\left(\left|\underset{g\in G}{\inf} \|g\cdot X-t_0\|^2-\underset{g\in G}{\inf} \|g\cdot X-x\|^2 \right|\right) \nonumber \\
&\leq \|x-t_0\| \left( 2\|t_0\|+\|x-t_0\|+\E(\|2X\|)\right).
\label{variationFb}
\end{align}

With Equations~\eqref{delta} and~\eqref{variationFb}, for all $x\in B(t_0,\delta)$ we have $F(x)>F(a_\star t_0)$. No point in that ball mapped in the quotient space is a Fréchet mean of $[X]$. So $\delta$ is a lower bound of the consistency bias. Now by using the fact that $\E(\|X\|)\leq \sqrt{\|t_0\|^2+\sigma^2}$, we get:
$2|F(t_0)-F(x)|\leq 2\|x-t_0\|\times \|t_0\|\left(1+\sqrt{1+\sigma^2/\|t_0\|^2}\right)+\|x-t_0\|^2
$. This proves that $\delta_\star$ is also a lower bound of the consistency bias.\qquad \end{proof}

$\delta_\star$ is smaller than $\delta$, but the variability of $X$ intervenes in $\delta_\star$. Therefore we propose to study the asymptotic behaviour of $\delta_\star$ when the variability tends to infinity. We have the following proposition:
\begin{proposition}
\label{sigmainfini}
Under the hypotheses of \Cref{lb}, we write $X=t_0+\sigma \epsilon$, with $\E(\epsilon)=0$, and $\E(\|\epsilon\|^2)=1$ and note $\nu=\E(\sup_{g\in G} \psh{g\epsilon}{t_0/\|t_0\|})\in (0,1]$, we have that:
\begin{equation*}
\delta_\star\underset{\sigma\to +\infini}{\sim} \sigma (\sqrt{1+\nu^2}-1),
\end{equation*}

\end{proposition}
In particular, the consistency bias explodes when the variability of $X$ tends to infinity.
\begin{proof}
First, let us prove that that $\nu\in (0,1]$ under the condition~\eqref{condition}. We have $\nu\geq \E(\psh{\epsilon}{t_0/\|t_0\|}=0$. By a \textit{reductio ad absurdum}: if $\nu=0$, then $\underset{g\in G}{\sup} \psh{g\epsilon}{t_0}=\psh{\epsilon}{t_0}$ almost surely. We have then almost surely:
$\psh{X}{t_0} \leq  \sup_{g\in G} \psh{gX}{t_0}\leq  \|t_0\|^2 +\sup_{g\in G} \sigma \psh{g\epsilon}{t_0}= \|t_0\|^2 +\sigma \psh{\epsilon}{t_0}\leq \psh{X}{t_0},
$
which is in contradiction with~\eqref{condition}. Besides $\nu \leq \E(\|\epsilon\|)\leq \sqrt{ \E\|\epsilon\|^2}=1$
 
Second, we exhibit equivalent of the terms in equation~\eqref{lowerbound} when $\sigma\to+\infini$:
\begin{equation}
2\|t_0\|\left(1+\sqrt{1+\sigma^2/\|t_0\|^2}\right)\sim 2\sigma.
\label{b}
\end{equation}
Now by definition of $a_\star$ in Equation~\eqref{astar} and the decomposition of $X=t_0+\sigma \epsilon$ we get:
\begin{align}
\|t_0\|(a_\star-1)&=\frac{1}{\|t_0\|}\E\left( \underset{g\in G}{\sup} ( \psh{g\cdot t_0}{t_0}+\psh{g\cdot \sigma \epsilon}{t_0})\right)-\|t_0\|\nonumber\\
\|t_0\|(a_\star-1)&\leq  \frac{1}{\|t_0\|}\E\left( \underset{g\in G}{\sup} \psh{g\cdot \sigma \epsilon}{t_0}\right)=\sigma \nu \label{majV}\\
\|t_0\|(a_\star-1) &\geq  \frac{1}{\|t_0\|}\E\left( \underset{g\in G}{\sup} \psh{g\cdot \sigma \epsilon}{t_0}\right)-2\|t_0\|=\sigma\nu-2\|t_0\|, \label{minV}
\end{align}
The lower bound and the upper bound of $\|t_0\|(a_\star-1)$ found in~\eqref{majV} and~\eqref{minV} are both equivalent to $\sigma \nu$, when $\sigma \to +\infini$. Then the constant term of the quadratic Equation~\eqref{lowerbound} has an equivalent:
\begin{equation}
-\|t_0\|^2(a_\star-1)^2 \sim - \sigma^2\nu^2.
\label{c}
\end{equation}
Finallye if we solve the quadratic Equation~\eqref{lowerbound}, we write $\delta_\star$ as a function of the coefficients of the quadratic equation~\eqref{lowerbound}. We use the equivalent of each of these terms thanks to equation~\eqref{b} and~\eqref{c}, this proves \cref{sigmainfini}. 
\end{proof}

\begin{remark}
Thanks to inequality~\eqref{minV}, if $\frac{\|t_0\|}{\sigma}<\frac{\nu}{2}$, then $\|t_0\|^2(1-a_\star)^2\geq (\sigma \nu-2\|t_0\|)^2$, then we write $\delta_\star$ as a function of the coefficients of Equation~\eqref{lowerbound}, we obtain a lower bound of the inconsistency bias as a function of $\|t_0\|$, $\sigma$ and $\nu$ for $\sigma> 2\|t_0 \|/\nu$:
\begin{equation*}
\frac{\delta_\star}{\|t_0\|} \geq -(1+\sqrt{1+\sigma^2/\|t_0\|^2})+\sqrt{(1+\sqrt{1+\sigma^2/\|t_0\|^2})^2+(\sigma \nu/\|t_0\|-2)^2}.
\end{equation*}

\end{remark}

Although the constant $\nu$ intervenes in this lower bound, it is not an explicit term. We now explicit its behaviour depending on $t_0$. We remind that:
\begin{equation*}
\nu=\frac{1}{\|t_0\|}\E\left(\underset{g\in G}{\sup} \psh{g\epsilon}{t_0}\right).
\end{equation*}
To this end, we first note that the set of fixed points under the action of $G$ is a closed linear space, (because we can write it as an intersection of the kernel of the continuous and linear functions: $x\mapsto g\cdot x-x$ for all $g\in G$). We denote by $p$ the orthogonal projection on the set of fixed points $\mbox{Fix}(M)$. Then for $x\in M$, we have: $\mbox{dist}(x,\mbox{Fix}(M))=\|x-p(x)\|$.  Which yields:
\begin{equation}
\psh{g\epsilon}{ t_0}=\psh{g\epsilon}{ t_0-p(t_0)}+\psh{\epsilon}{p(t_0)}.
\label{projg}
\end{equation}
The right hand side of Equation~\eqref{projg} does not depend on $g$ as $p(t_0)\in \mbox{Fix}(M)$. Then:
\begin{equation*}
\|t_0\|\nu=\E\left(\underset{g\in G}{\sup} \psh{g \epsilon}{t_0-p(t_0)}\right) +\psh{\E(\epsilon)}{p(t_0)}.
\end{equation*}
Applying the Cauchy-Schwarz inequality and using $\E(\epsilon)=0$, we can conclude that:
\begin{equation}
\nu\leq  \frac{1}{\|t_0\|}\mbox{dist}(t_0,\mbox{Fix}(M)) \E(\|\epsilon\|)=\mbox{dist}(t_0/\|t_0\|,\mbox{Fix}(M)) \E(\|\epsilon\|).
\label{nuup}
\end{equation}
This leads to the following comment: our lower bound of the consistency bias is smaller when our normalized template $t_0/\|t_0\|$ is closer to the set of fixed points.

\subsection{Upper bound of the consistency bias}
\label{subsec:up}
In this Section, we find a upper bound of the consistency bias. More precisely we have the following Theorem:
 \begin{proposition}
 \label{ub}
 Let $X$ be a random variable in $M$, such that $X=t_0+\sigma \epsilon$ where $\sigma>0$, $\E(\epsilon)=0$ and $\E(||\epsilon||^2)=1$. We suppose that $[m_\star]$ is a Fréchet mean of $[X]$. Then we have the following upper bound of the quotient distance between the orbit of the template $t_0$ and the Fréchet mean of $[X]$:
\begin{equation}
d_ Q([m_\star],[t_0])\leq \sigma\nu(m_*-m_0)+
\sqrt{\sigma^2\nu(m_*-m_0)^2+2\mbox{dist}(t_0,\mbox{Fix}(M))\sigma\nu(m_*-m_0)},
\label{equp}
\end{equation}

where we have noted $\nu(m)=\E(\sup_g\psh{g\epsilon}{m/\|m\|})\in [0,1]$ if $m\neq 0$ and $\nu(0)=0$, and $m_0$ the orthogonal projection of $t_0$ on $Fix(M)$.
\end{proposition}

Note that we made no hypothesis on the template in this proposition. We deduce from Equation~\eqref{equp} that $d_Q([m_\star],[t_0])\leq \sigma+\sqrt{\sigma^2+2\sigma\mbox{dist}(t_0,\mbox{Fix}(M))}$ is a $O(\sigma)$ when $\sigma\to \infini$, but a $O(\sqrt{\sigma})$ when $\sigma\to 0$, in particular the consistency bias can be neglected when $\sigma$ is small.
\begin{proof}
First we have: 
\begin{equation}
F(m_\star)\leq F(t_0)=\E(\inf_g ||t_0-g(t_0+\sigma \epsilon)||^2)\leq \E(||\sigma \epsilon||^2)=\sigma^2.
\label{ineg1}
\end{equation}

Secondly we have for all $m\in M$, (in particular for $m_\star$):
\begin{align}
F(m)&=&\E(\inf_{g}(\|m-gt_0\|^2+\sigma^2\|\epsilon\|^2-2\langle g\sigma\epsilon,m-gt_0\rangle)) \nonumber\\
&\geq& d_Q([m],[t_0])^2+\sigma^2-2\E(\sup_g\langle \sigma \epsilon,gm\rangle)
\label{ineg2}.
\end{align}

With Inequalities~\eqref{ineg1} and~\eqref{ineg2} one gets:
\begin{equation*}
d_Q([m_*],[t_0])^2\leq 2\E(\sup_g\psh{\sigma\epsilon}{gm_\star})= 2\sigma\nu(m_\star)||m_\star||, \label{d2v}
\end{equation*}
note that at this point, if $m_\star=0$ then $\E(\sup_g\psh{\sigma\epsilon}{gm_\star})=0$ and $\nu(m_\star)=0$ although Equation~\eqref{d2v} is still true even if $m_\star=0$.
Moreover with the triangular inequality applied at $[m_\star],\: [0]$ and $[t_0]$, one gets: $ \|m_\star\|\leq \|t_0\|+d_Q([m_\star],[t_0])$ and then:
\begin{equation}
d_Q([m_*],[t_0])^2\leq 2\sigma\nu(m_\star)(d_Q([m_*],[t_0])+\|t_0\|).
\label{trinome}
\end{equation}
We can solve inequality~\eqref{trinome} and we get:

\begin{equation}
    d_ Q([m_\star],[t_0])\leq \sigma\nu(m_\star)+
    \sqrt{\sigma^2\nu(m_\star)^2+2\|t_0\|\sigma\nu(m_\star)},
    \label{intem}
\end{equation}
We note by $F_X$ instead of $F$ the variance in the quotient space of $[X]$, and we want to apply inequality~\eqref{intem} to $X-m_0$. As $m_0$ is a fixed point:
\begin{equation*}
F_X(m)= \E\left( \underset{g\in G}{\inf} \|X-m_0-g\cdot (m -m_0) \|^2\right)=F_{X-m_0}(m-m_0)
\end{equation*}
Then $m_\star$ minimises $F_X$ if and only if $m_\star-m_0$ minimises $F_{X-m_0}$. We apply Equation~\eqref{intem} to $X-m_0$, with $\E(X-m_0)=t_0-m_0$ and $[m_\star-m_0]$ a Fréchet mean of $[X-m_0]$. We get:
\begin{equation*}
d_Q([m_\star-m_0], [t_0-m_0]) \leq \sigma\nu(m_*-m_0)+\sqrt{\sigma^2\nu(m_*-m_0)^2+2\|t_0-m_0\|\sigma\nu(m_*-m_0)}.
\end{equation*}
Moreover $d_ Q([m_\star],[t_0])=d_Q([m_\star-m_0],[t_0-m_0])$, which concludes the proof. \qquad
\end{proof}

\subsection{Empirical Fréchet mean}
In practice, we never compute the Fréchet mean in quotient space, only the empirical Fréchet mean in quotient space when the size of a sample is supposed to be large enough. If the empirical Fréchet in the quotient space means converges to the Fréchet mean in the quotient space then we can not use these empirical Fréchet mean in order to estimate the template. In~\cite{bha}, it has been proved that the empirical Fréchet mean converges to the Fréchet mean with a $\frac{1}{\sqrt n}$ convergence speed, however the law of the random variable is supposed to be included in a ball whose radius depends on the geometry on the manifold. Here we are not in a manifold, indeed the quotient space contains singularities, moreover we do not suppose that the law is necessarily bounded. However in~\cite{zie} the empirical Fréchet means is proved to converge to the Fréchet means but no convergence rate is provided.


We propose now to prove that the quotient distance between the template and the empirical Fréchet mean in quotient space have an lower bound which is the asymptotic of the one lower bound of the consistency bias found in~\eqref{lowerboundb}. Take $X, X_1,\tp,X_n$ independent and identically distributed (with $t_0=\E(X)$ not a fixed point). We define the empirical variance of $[X]$ by:
\begin{equation*}
m\in M\mapsto F_n(m)=\frac{1}{n} \somm{i=1}{n} d_Q([m],[X_i])^2=\frac{1}{n} \somm{i=1}{n} \underset{g\in G}{\inf} \|m-g\cdot X_i\|^2,
\end{equation*}
and we say that $[m_{n\star}]$ is a empirical Fréchet mean of $[X]$ if $m_{n\star}$ is a global minimiser of $F_n$.

\begin{proposition}
Let $X,X_1,\tp,X_n$ independent and identically distributed random variables, with $t_0=\E(X)$. Let be $[m_{n\star}]$ be an empirical Fréchet mean of $[X]$. Then $\delta_n$ is a lower bound of the quotient distance between the orbit of the template and $[m_{n\star}]$, where $\delta_{n}$ is the unique positive solution of:
\begin{equation*}
\delta^2+2\left(||t_0||+ \frac{1}{n}\somm{i=1}{n} \|X_i\|\right)\delta -\|t_0\|^2(a_{n\star}-1)^2=0.
\end{equation*}
$a_{n\star}$ is defined like $a_\star$ in \cref{subsec:gbias} by:
\begin{equation*}
a_{n\star} = \frac{\frac1n\somm{i=1}n  \underset{g\in G}{\sup} \psh{g\cdot X_i}{t_0} }{\|t_0\|^2}.
\end{equation*}

We have that $\delta_n\to \delta$ by the law of large numbers.
\end{proposition}

The proof is a direct application of \cref{lb}, but applied to the empirical law of $X$ given by the realization of $X_1,\tp, X_n$.
\subsection{Examples}
In this Subsection, we discuss, in some examples, the application of \cref{theogeneral} and see the behaviour of the constant $\nu$. This constant intervened in lower bound of the consistency bias.
\subsubsection{Action of translation on \texorpdfstring{$L^2(\R/\Z)$}{L2(R/Z}}
We take an orbit $O=[f_0]$, where $f_0$ $\in \Cc^2(\R/\Z)$, non constant. We show easily
that $O$ is a manifold of dimension $1$ and the tangent space at $f_0$
 is\footnote{Indeed $\varphi: \begin{matrix} 
]-\frac12,\frac12[&\to & O\\
t&\mapsto &f_0(. - t)
\end{matrix}$ is a local parametrisation of $O$: $f_0=\varphi(0)$, and we check that:
$\underset{x\to 0}{\lim} \|\varphi(x)-\varphi(0)-xf_0'\|_{L^2}=0$ with Taylor-Lagrange inequality at the order 2. As a conclusion $\varphi$ is differentiable at $0$, and it is an immersion (since $f_0'\neq 0$), and $D_0\varphi:x\mapsto xf_0'$, then $O$ is a manifold of dimension $1$ and the tangent space of $O$ at $f_0$ is: $
T_{f_0}O=D_0\varphi(\R)=\R f_0'$.}
$\R f'_0$. Therefore a sufficient condition on $X$ such that $\E(X)=f_0$ to have an inconsistency is: $\P(X\notin f_0'^\perp)>0$ according to \cref{biasmanifold}. Now if we denote by $\fc$ the constant function on $\R/\Z$ equal to $1$. We have in this setting: that the set of fixed points under the action of $G$ is the set of constant functions: $\mbox{Fix}(M)=\R \fc$ and:
\begin{equation*}
\mbox{dist}(f_0,\mbox{Fix}(M))= \|f_0-\psh{f_0}{\fc}\fc\|=\sqrt{\int_0^1 \left(f_0(t)-\int_0^1 f_0(s)ds\right)^2dt}.
\end{equation*}
This distance to the fixed points is used in the upper bound of the constant $\nu$ in Equation~\eqref{nuup}.
Note that if $f_0$ is not differentiable, then $[f_0]$ is not necessarily a manifold, and \eqref{biasmanifold} does not apply. However \cref{denseorbit} does: if $f_0$ is not a constant function, then $[f_0]\setminus \{f_0\}$ is dense in $[f_0]$. Therefore as soon as the support of $X$ contains a ball around $f_0$, there is an inconsistency.


\subsubsection{Action of discrete translation on \texorpdfstring{$\R^{\Z/\N\Z}$}{R^(Z/NZ)}}

We come back on \cref{ex2}, with $D=1$ (discretised signals). For some signal $t_0$, $\nu$ previously defined is:
\begin{equation*}
\nu=\frac{1}{\|t_0\|} \E\left(\underset{\tau\in \Z/\N\Z}{\max} \psh{\epsilon}{\tau \cdot t_0}\right).
\end{equation*}
Therefore if we have a sample of size $I$ of $\epsilon$ \textit{iid}, then:
\begin{equation*}
\nu=\frac{1}{\|t_0\|}\underset{I\to +\infini}{\lim} \frac{1}{I} \somm{i=1}{I} \underset{\tau_i \in \Z/N\Z}{\max} \psh{\epsilon_i}{\tau_i \cdot t_0},
\end{equation*}
By an exhaustive research, we can find the $\tau_i$'s which maximise the dot product, then with this sample and $t_0$ we can approximate $\nu$. We have done this approximation for several signals $t_0$ on \cref{diffnu}. According the previous results, the bigger $\nu$ is, the more important the lower bound of the consistency bias is. We remark that the $\nu$ estimated is small, $\nu \ll 1$ for different signals.
\begin{figure}[!ht]\centering\includegraphics[clip=true,trim=2.6cm 7.5cm 2.5cm 6cm,scale=0.35]{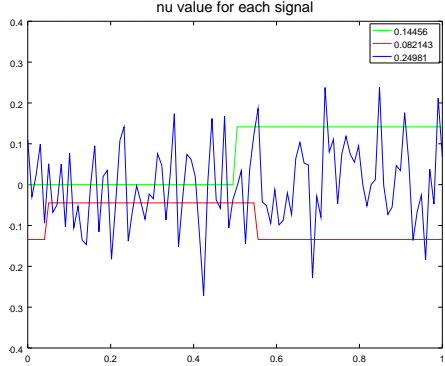}\caption[Different signals and their $\nu$ approximated]{Different signals and their $\nu$ approximated with a sample of size $10^3$ in $\R^{\Z/100\Z}$. $\epsilon$ is here a Gaussian noise in $\R^{\Z/100\Z}$, such that $\E(\epsilon)=0$ and $\E(\|\epsilon\|^2)=1$. For instance the blue signal is a signal defined randomly, and when we approximate the $\nu$ which corresponds to that $t_0$ we find $\simeq 0.25$.}\label{diffnu}\end{figure}

\subsubsection{Action of rotations on $\R^n$}

Now we consider the action of rotations on $\R^n$ with a Gaussian noise. Take $X\sim \Nc(t_0,s^2 Id_n)$ then the variability of $X$ is $n s^2$, then $X$ has a decomposition: $X=t_0+\sqrt{n}s \epsilon$ with $\E(\epsilon)=0$ and $\E(\|\epsilon\|^2)=1$. According to \cref{sigmainfini} we have by noting $\delta_\star$ the lower bound of the consistency bias when $s\to \infini$:
\begin{equation*}
    \frac{\delta_\star}{s}\to \sqrt{n}(-1+\sqrt{1+\nu^2}).
\end{equation*}
Now $\nu=\E(\sup_{g\in G} \psh{g\epsilon}{t_0)}/\|t_0\|=\E(\|\epsilon\|) \to 1$ when $n$ tends to infinity (expected value of the Chi distribution) we have that for $n$ large enough:

\begin{equation*}
     \lim_{s\to\infini}\frac{\delta_\star}{s} \simeq \sqrt{n}(\sqrt 2-1).
\end{equation*}
We compare this result with the exact computation of the consistency bias (noted here $CB$) made by Miolane et al.~\cite{mio2}, which writes with our current notations:
\begin{equation*}
    \lim_{s\to \infini} \frac{\mbox{CB}}{s} =\sqrt2 \frac{\Gamma((n+1)/2)}{\Gamma(n/2)}.
\end{equation*}
Using a standard Taylor expansion on the Gamma function, we have that for $n$ large enough:
\begin{equation*}
    \lim_{s\to \infini} \frac{\mbox{CB}}{s} \simeq \sqrt{n}.
\end{equation*}
As a conclusion, when the dimension of the space is large enough our lower bound and the exact computation of the bias have the same asymptotic behaviour. It differs only by the constant $\sqrt 2-1\simeq 0.4$ in our lower bound, $1$ in the work of Miolane et al.~\cite{mio}.

\section{Fréchet means top and quotient spaces are not consistent when the template is a fixed point}
\label{sec:fixedpoint}
In this Section, we do not assume that the top space $M$ is a vector space, but rather a manifold. We need then to rewrite the generative model likewise: let $t_0\in M$, and $X$ any random variable of $M$ such as $t_0$ is a Fréchet mean of $X$. Then $Y=S\cdot X$ is the observed variable where $S$ is a random variable whose value are in $G$. In this Section we make the assumption that the template $t_0$ is a fixed point under the action of $G$.

\subsection{Result}
Let $X$ be a random variable on $M$ and define the variance of $X$ as:
\begin{equation*}
E (m) = \E(d_M(m, X)^2 ).
\end{equation*}
We say that $t_0$ is a Fréchet mean of $X$ if $t_0$ is a global minimiser of the variance $E$. We prove the following result:
\begin{theorem}
\label{theo1}
Assume that $M$ is a complete finite dimensional Riemannian manifold and that
$d_M$ is the geodesic distance on $M$. Let $X$ be a random variable on $M$, with $\E(d(x , X)^2 ) <+\infini$ for some $x\in M$. We assume that $t_0$ is a fixed point and a Fréchet mean of $X$ and that $\P(X \in C(t_0)) = 0$ where $C(t_0)$ is the cut locus of $t_0$. Suppose that there exists a point in the support of $X$ which is not a fixed point nor in the cut locus of $t_0$. Then $[t_0]$ is not a Fréchet mean of $[X]$.
\end{theorem}

The previous result is finite dimensional and does not cover interesting infinite dimensional setting concerning curves for instance. However, a simple extension to the previous result can be stated when $M$ is a Hilbert vector space since then the space is flat and some technical problems like the presence of cut locus point do not occur.

\begin{theorem}
\label{theo2}
Assume that $M$ is a Hilbert space and that $d_M$ is given by the Hilbert norm on $M$. Let $X$ be a random variable on $M$, with $\E(\|X\|^2 ) <+\infini$. We assume that $t_0=\E(X)$. Suppose that there exists a point in the support of the law of $X$ that is not a fixed point for the action of $G$. Then $[t_0]$ is not a Fréchet mean of $[X]$.
\end{theorem}

Note that the reciprocal is true: if all the points in the support of the law of $X$ are fixed points, then almost surely, for all $m\in M$ and for all $g\in G$ we have:
\begin{equation*}
d_M(X,m)=d_M(g\cdot X, m)=d_Q([X],[m]). 
\end{equation*}
Up to the projection on the quotient, we have that the variance of $X$ is equal to the variance of $[X]$  in $M/G$, therefore $[t_0]$ is a Fréchet mean of $[X]$ if and only if $t_0$ is a Fréchet mean of $X$. There is no inconsistency in that case.

\begin{example}
\Cref{theo2} covers the interesting case of the Fisher Rao metric on functions:
\begin{equation*}
\mathcal F = \{f :
[0, 1]\to \R \quad | \quad f \mbox{ is absolutely continuous} \}.
\end{equation*}
Then considering for $G$ the
group of smooth diffeomorphisms $\gamma$ on $[0, 1]$ such that $\gamma(0) = 0$ and $\gamma(1) = 1$, we have a
right group action $G \times \Fc\to \Fc$ given by $\gamma\cdot f = f\circ \gamma $. The Fisher Rao metric is built as a pull back metric of the $ L^2 ([0, 1], \R)$ space through the map $Q : \Fc \to L^2$ given
by: $Q(f) = \dot f/\sqrt{|\dot f|}.$ 
This square root trick is often used, see for instance~\cite{kur}. Note that in this case, $Q$ is a bijective mapping with inverse given by $q \mapsto f$ with $f(t) =\int_0^t q(s)|q(s)|ds$. We can define a group action on $M=  L^2$ as:
$\gamma\cdot q =  q\circ  \gamma \sqrt{\dot{\gamma}}
$,
for which one can check easily by a change of variable that:
\begin{equation*}
\|\gamma\cdot q -\gamma \cdot q' \|^2 = \| q\circ \gamma \sqrt{\dot \gamma} -q'\circ \gamma \sqrt{\dot \gamma}\|^2=  \|q-q'\|^2.
\end{equation*}
So up to the mapping $Q$, the Fisher Rao metric on curve corresponds to the situation $M$ where \cref{theo2} applies. Note that in this case the set of fixed points under the action of $G$ corresponds in the space $\mathcal F$ to constant functions.
\end{example}

We can also provide an computation of the consistency bias in this setting:
\begin{proposition}
\label{cbfp}
Under the assumptions of \cref{theo2}, we write $X=t_0+\sigma \epsilon$ where $t_0$ is a fixed point, $\sigma>0$, $\E(\epsilon)=0$ and $\E(\|\epsilon\|^2)=1$, if there is a Fréchet mean of $[X]$, then the consistency bias is linear with respect to $\sigma$ and it is equal to:
\begin{equation*}
  \sigma  \underset{\|v\|=1}{\sup} \E( \underset{g\in G}{\sup} \psh{v}{g\cdot \epsilon}).
\end{equation*}
\end{proposition}
\begin{proof}
For $\lambda>0$ and $\|v\|=1$, we compute the variance $F$ in the quotient space of $[X]$ at the point $t_0+\lambda v$. Since $t_0$ is a fixed point we get:
\begin{equation*}
F(t_0+\lambda v)=\E(\underset{g\in G}{\inf} \|t_0+\lambda v-gX\|^2)=\E(\|X\|^2)-\|t_0\|^2-2\lambda \E(\sup_g \psh{v}{g(X-t_0)})+\lambda^2.
\end{equation*}
Then we minimise $F$ with respect to $\lambda$, and after we minimise with respect to $v$ (with $\|v\|=1$). Which concludes.
\end{proof}

\subsection{Proofs of these theorems}
\subsubsection{Proof of \cref{theo1}}
We start with the following simple result, which aims to differentiate the variance of $X$. This classical result (see~\cite{pen} for instance) is proved in \cref{sec:plemmaf} in order to be the more self-contained as possible:

\begin{lemma}
\label{lemmaf} Let $X$ a random variable on $M$ such that $\E(d(x , X)^2 ) <+\infini$ for some $x\in M$. Then the variance $m \mapsto E (m) = \E(d_M(m, X) ^2 )$ is a continuous function which is differentiable at any point $m\in M$ such that $\P(X \in C(m)) = 0$ where $C(m)$ is the cut locus of $m$. Moreover at such point one has:
\begin{equation*}
\nabla E (m) = -2\E(\log_m (X)),
\end{equation*}
where $\log_m : M \setminus C(m) \to T_m M$ is defined for any $x\in M \setminus C(m)$ as the unique $u\in T_m M$
such that $exp_m (u) = x$ and $\|u\|_m = d_M(x, m)$.
\end{lemma}

We are now ready to prove \cref{theo1}.

\begin{proof} (of \cref{theo1}) Let $m_0$ be a point in the support of $M$ which is not a fixed point and not in the cut locus of $t_0$. Then there exists $g_0\in G$ such that $m_1 = g_0 m_0\neq m_0$. Note that since $x\mapsto g_0x$ is a symmetry (the distance is equivariant under the action of $G$) have that $m_1 = g_0 m_0\notin C(g_0 t_0) = C(t_0)$ ($t_0$ is a fixed point under the action of $G$). Let $v_0 = \log_{t_0}(m_0 )$ and $v_1 = \log_{t_0}(m_1 )$. We have  $v_0\neq v_1$ and since $C(t_0)$ is closed and the $\log_{t_0}$ is continuous application on $M\setminus C(t_0)$ we have:
\begin{equation*}
\underset{\epsilon \to 0}{\lim}\frac{1}{\P(X\in B(m_0,\epsilon))}\E(\fc_{X\in B(m_0,\epsilon)} \log_{t_0}(X))= v_0.
\end{equation*} (we use here the fact that since $m_0$ is in the support of the law of $X$, $\P(X \in B(m_0 ,\epsilon )) > 0$ for any $\epsilon > 0$ so that the denominator does not vanish and the fact that since M is a complete manifold,
it is a locally compact space (the closed balls are compacts) and $\log_{t_0}$ is locally bounded).
Similarly:
\begin{equation*}
\underset{\epsilon \to 0}{\lim} \frac{1}{\P(X\in B(m_0,\epsilon))}\E(\fc_{X\in B(m_0,\epsilon)} \log_{t_0}(g_0X))= v_1.
\end{equation*}
Thus for sufficiently small $\epsilon > 0$ we have (since $v_0\neq v_1$):
\begin{equation}
\E(\log_{t_0} (X)\fc_{X\in B(m_0 , \epsilon) }) \neq \E(\log_{t_0} (g_0 X)\fc_{X\in B(m_0 ,\epsilon )}).
\label{differentlog}
\end{equation}
By using using a \textit{reductio ad absurdum}, we suppose that $[t_0]$ is a Fréchet mean of $[X]$ and we want to find a contradiction with~\eqref{differentlog}. In order to do that we introduce simple functions as the function $x\mapsto \fc_{x\in B(m_0 , \epsilon) }$ which intervenes in Equation~\eqref{differentlog}. Let $s : M\to G$ be a simple function (i.e. a measurable function with finite number of values in $G$). Then $x\mapsto h(x) = s(x)x$ is a measurable function\footnote{Indeed if: $s = \somm{i=1}{n} g_i \fc_{A_i}$ where $(A_i)_{1\leq i\leq n}$ is a partition of $M$ (such that the sum is always defined). Then for any Borel set $B \subset M$ we have:
$h^{-1}(B) =\underset{i=1}{\overset{n}{\bigcup}} g_i^{-1}(B)\cap A_i$ is a measurable set since $x \mapsto  g_i x$ is a measurable function.}. Now, let $E_s (x) = \E(d(x, s(X)X)^2 )$ be the variance of the variable $s(X)X$. Note that (and this is the main point):
\begin{equation*}
\forall g\in G\qquad d_M(t_0 , x) = d_M(gt_0 , gx) = d_M(t_0 , gx) =  d_Q([t_0],[x]),
\end{equation*}
we have: $E_s (t_0) = E (t_0)$. Assume now that $[t_0]$ a Fréchet mean for $[X]$ on the quotient space and let us show that $E_s$ has a global minimum at $t_0$. Indeed for any $m$, we have:
\begin{equation*}
E_s(m)= \E(d_M(m,s(X)X)^2)\geq  \E(d_Q( [m],[X])^2)\geq \E(d_Q([t_0],[X])^2)= E_s(t_0).
\end{equation*}
Now, we want to apply \cref{lemmaf} to the random variables $s(X)X$ and $X$ at the point $t_0$. Since we assume that $X\notin  C(t_0)$ almost surely and $X\notin C(t_0)$ implies $s(X)X\notin C(t_0)$ we get $\P(s(X)X \in C(t_0)) = 0$ and the \cref{lemmaf} applies.
As $t_0$ is a minimum, we already know that the differential of $E_s$ (respectively $E$) at $t_0$ should be zero. We get:
\begin{equation}
\E(\log_{t_0}(X)) = \E(\log_{t_0}(s(X)X)) = 0.
\label{se}
\end{equation}
Now we apply Equation~\eqref{se} to a particular simple function defined by $s(x) = g_0 \fc_{x\in B(m_0,\epsilon)}+ e_G \fc_{x\notin B(m_0,\epsilon)}$.  We split the two expected values in~\eqref{se} into two parts:
\begin{equation}
\label{eq1}
    \E(\log_{t_0}(X) \fc_{X\in B(m_0,\epsilon)})+ \E(\log_{t_0}(X) \fc_{X\notin B(m_0,\epsilon)})=0,
\end{equation}
\begin{equation}
\label{eq2}
    \E(\log_{t_0}(g_0 X) \fc_{X\in B(m_0,\epsilon)})+ \E(\log_{t_0}( X) \fc_{X\notin B(m_0,\epsilon)})=0.
\end{equation}
By substrating~\eqref{eq1} from~\eqref{eq2}, one gets:
\begin{equation*}
\E(\log_{t_0}(X)\fc_{X\in B(m_0,\epsilon)}) = \E(\log_{t_0}(g_0X)\fc_{X\in B(m_0,\epsilon)}),
\end{equation*}
which is a contradiction with~\eqref{differentlog}. Which concludes.\qquad
\end{proof}

\subsubsection{Proof of \cref{theo2}}

\begin{proof}
The extension to \cref{theo2} is quite straightforward. In this setting many things are now explicit since $d(x, y) = \|x-y\|$ , $\nabla_x d(x, y)^2 = 2(x-y)$, $\log_x (y) = y-x$ and the cut locus is always empty. It is then sufficient to go along the previous proof and to change the quantity accordingly. Note that the local compactness of the space is not true in infinite dimension. However this was only used to prove that the log was locally bounded but this last result is trivial in this setting.\qquad
\end{proof}

\section{Conclusion and discussion}

In this article, we exhibit conditions which imply that the template estimation with the Fréchet mean in quotient space is inconsistent. These conditions are rather generic. As a result, without any more information, \textit{a priori} there is inconsistency. 
The behaviour of the consistency bias is summarized in \cref{tab:con}. Surely future works could improve these lower and upper bounds. 
\begin{table}[!ht]
\caption{Behaviour of the consistency bias with respect to $\sigma^2$ the variability of $X=t_0+\sigma \epsilon$. The constants $K_i$'s depend on the kind of noise, on the template $t_0$ and on the group action.}
\label{tab:con}
\begin{tabular}{|p{5cm}|p{4.6cm}|p{4.6cm}|}
\hline
 Consistency bias : $CB$ & $G$ is any group & Supplementary properties for $G$ a finite group \\ \hline
Upper bound of $CB$ & $CB \leq\sigma +2\sqrt{\sigma^2+K_1\sigma}$ (\cref{ub})
& $CB \leq K_2\sigma$ (\cref{alla}) \\ \hline
Lower bound of $CB$ for $\sigma\to \infini$ when the template is not a fixed point &
\multicolumn{2}{c|}{ $CB\geq L \underset{\sigma\to \infini}{\sim} K_3\sigma$
(\cref{sigmainfini})} \\ \hline
Behavior of CB for $\sigma\to 0$ when the template is not a fixed point & $CB \leq U\underset{\sigma\to 0}\sim K_4\sqrt{\sigma}$
& $CB\underset{0}{=}o(\sigma^k)$, $\forall k\in \N$ in the \cref{subsec:Exemple}, can we extend this result for finite group? \\\hline
$CB$ when the template is a fixed point & \multicolumn{2}{c|}{$CB=\sigma \underset{\|v\|=1}\sup \E( \sup_{g\in G} \psh{v}{g\epsilon})$ (\cref{cbfp})} \\ \hline
\end{tabular}
\end{table}

In a more general case: when we take an infinite-dimensional vector space quotiented by a non isometric group action, is there always an inconsistency? An important example of such action is the action of diffeomorphisms. Can we estimate the consistency bias?
In this setting, one estimates the template (or an atlas), but does not exactly compute the Fréchet mean in quotient space, because a regularization term is added.
In this setting, can we ensure that the consistency bias will be small enough to estimate the original template? Otherwise, one has to reconsider the template estimation with stochastic algorithms as in~\cite{all2} or develop new methods.

\appendix \section{Proof of theorems for finite groups' setting}

\subsection{Proof of \cref{difftheo}: differentiation of the variance in the quotient space}
\label{subsec:prop}
In order to show \cref{difftheo} we proceed in three steps. First we see some following properties and definitions which will be used. Most of these properties are the consequences of the fact that the group $G$ is finite. Then we show that the integrand of $F$ is differentiable. Finally we show that we can permute gradient and integral signs.
\begin{enumerate}
\item \label{singnull}
The set of singular points in $\R^n$, is a null set (for the Lebesgue's measure), since it is equal to:
\begin{equation*}
\underset{g\neq e_G}{\bigcup} \ker(x\mapsto g\cdot x-x),
\end{equation*}
a finite union of strict linear subspaces of $\R^n$ thanks to the linearity and effectively of the action and to the finite group.
\item If $m$ is regular, then for $g,\:g'$ two different elements of $G$, we pose:
\begin{equation*}
H(g\cdot m,g'\cdot m)=\{x\in \R^n,\: \|x-g\cdot m\|=\|x-g'\cdot m\|\}.
\end{equation*}
Moreover $H(g\cdot m,g'\cdot m)=(g\cdot m-g'\cdot m)^\perp$ is an hyperplane.
\item \label{defAA} For $m$ a regular point we define the set of points which are equally distant from two different points of the orbit of $m$:
\begin{equation*}
A_{m}=\underset{g\neq g'}{\bigcup} H(g\cdot m,g'\cdot m).
\end{equation*}
Then $A_{m}$ is a null set. For $m$ regular and $x\notin A_{m}$ the minimum in the definition of the quotient distance :
\begin{equation}
d_Q([m],[x])=\underset{g\in G}{\min} \|m-g\cdot x\|,
\label{deftau}
\end{equation} is reached at a unique $g\in G$, we call $g(x,m)$ this unique element.
\item \label{maxdot} By expansion of the squared norm: $g$ minimises $\|m-g\cdot x\|$ if and only if $g$ maximises $\psh{m}{g\cdot x}$.
\item
\label{gconstant}If $m$ is regular and $x\notin A_m$ then:
\begin{equation*}
\forall g\in G\setminus\{g(x,m)\},\: \|m-g(x,m)\cdot x\|<\|m-g \cdot x\|,
\end{equation*}
by continuity of the norm and by the fact that $G$ is a finite group, we can find $\alpha>0$, such that for $\mu\in B(m,\alpha)$ and $y\in B(x,\alpha)$:
\begin{equation}
\forall g \in G\setminus\{g(x,m)\} \: \|\mu-g(x,m)\cdot y\|<\|\mu-g\cdot y\|.
\label{munu}
\end{equation}
Therefore for such $y$ and $\mu$ we have:
\begin{equation*}
g(x,m)=g(y,\mu).
\label{tauconstant}
\end{equation*}
\item \label{dcone}
For $m$ a regular point, we define $Cone(m)$ the convex cone of $\R^n$:
\begin{align}
Cone(m)&= \{x\in \R^n \:/\: \forall g\in G\: \|x-m\|\leq \|x-g\cdot m\| \} \label{defPg}\\
&=\{x\in \R^n \:/\: \forall g\in G\: \psh{m}{x}\geq \psh{gm}{x}\}.\nonumber
\end{align} 
This is the intersection of $|G|-1$ half-spaces: each half space is delimited by $H(m,gm)$ for $g\neq e_G$ (see \cref{fig:Pg}). $Cone(m)$ is the set of points whose projection on $[m]$ is $m$, (where the projection of one point $p$ on $[m]$ is one point $g\cdot m$ which minimises the set $\{\|p-g\cdot m\|, \: g\in G\}$).

\item Taking a regular point $m$ allows us to see the quotient. For every point $x\in \R^n$ we have: $[x]\bigcap Cone(m)\neq \emptyset$, $card ([x]\bigcap Cone(m))\geq 2$ if and only if $x\in A_m$. The borders of the cone is $Cone(m)\setminus \mbox{Int}(Cone(m))=Cone(m)\cap A_m$ (we denote by $\mbox{Int}(A)$ the interior of a part $A$). Therefore $Q=\R^n/G$ can be seen like $Cone(m)$ whose border have been glued together.
\end{enumerate}

The proof of \cref{difftheo} is the consequence of the following lemmas. The first lemma studies the differentiability of the integrand, and the second allows us to permute gradient and integral sign. Let us denote by $f$ the integrand of $F$:

\begin{equation}
\forall \: m,\:x\in M \quad f(x,m)= \underset{g\in G}\min\|m-g\cdot x\|^2.
\label{defintegrande}
\end{equation}
Thus we have: $F(m)=\E( f(X,m))$. The $\min$ of differentiable functions is not necessarily differentiable, however we prove the following result:

\begin{lemma}
\label{lemme:devf}
Let $m_0$ be a regular point, if $x\notin A_{m_0}$ then $m\mapsto f(x,m)$ is differentiable at $m_0$, besides we have:
\begin{equation}
\devp{f}{m}(x,m_0)=2(m_0-g(x,m_0)\cdot x)
\label{devf}
\end{equation}
\end{lemma}

\begin{proof} If $m_0$ is regular and $x\notin A_{m_0}$ then we know from the \cref{gconstant} of the \cref{subsec:prop} that $g(x,m_0)$ is locally constant. Therefore around $m_0$, we have:
\begin{equation*}
f(x,m)=\|m-g(x,m_0)\cdot x\|^2,
\end{equation*}
which can differentiate with respect to $m$ at $m_0$. This proves the \cref{lemme:devf}.\qquad
\end{proof}

Now we want to prove that we can permute the integral and the gradient sign. The following lemma provides us a sufficient condition to permute integral and differentiation signs thanks to the dominated convergence theorem:
\begin{lemma} For every $m_0\in M$ we have the existence of an integrable function $\Phi:M\to \R^+$ such that:
\begin{equation}
\forall m\in B(m_0,1), \: \forall x\in M\quad |f(x,m_0)-f(x,m)|\leq \|m-m_0\|\Phi(x).
\label{variaB}
\end{equation}
\end{lemma}

\begin{proof}For all $g\in G$, $m\: \in M$ we have:
\begin{align*}
\| g\cdot x-m_0\|^2-\|g\cdot x-m\|^2&=\psh{m-m_0}{2g\cdot x-(m_0+m)}\\
&\leq  \|m-m_0\|\times \left(\|m_0+m\|+\|2x\|\right)\\
\underset{g\in G}{\min} \| g\cdot x-m_0\|^2&\leq \|m-m_0\|\left(\|m_0+m\|+\|2x\|\right)+\|g\cdot x-m\|^2\\
\underset{g\in G}{\min} \| g\cdot x-m_0\|^2&\leq \|m-m_0\|\left(\|m_0+m\|+\|2x\|\right)+\underset{g\in G}{\min}\|g\cdot x-m\|^2\\
\underset{g\in G}{\min} \| g\cdot x-m_0\|^2-\underset{g\in G}{\min}\|g\cdot x-m\|^2&\leq  \|m-m_0\|\left(2\|m_0\|+\|m-m_0\|+\|2x\|\right)
\end{align*}
By symmetry we get also the same control of $f(x,m)-f(x,m_0)$, then:
\begin{equation}
|f(x,m_0)-f(x,m)|\leq \|m_0-m\| \left(2\|m_0\|+\|m-m_0\|+\|2x\|\right) \label{variation}
\end{equation}
The function $\Phi$ should depend on $x$ or $m_0$, but not on $m$. That is why we take only $m\in B(m_0,1)$, then we replace $\|m-m_0\|$ by $1$ in \eqref{variation}, which concludes.\qquad
\end{proof}

\subsection{Proof of \cref{theo}: the gradient is not zero at the template}
\label{subsec:prooftheo}
To prove it, we suppose that $\nabla F(t_0)=0$, and we take the dot product with $t_0$:
\begin{equation}\psh{\nabla F(t_0)}{t_0}=2\E(\psh{X}{t_0} -\psh{g(X,t_0)\cdot X}{t_0})=0.
\label{nablaJx}
\end{equation}
The \cref{maxdot} of $(x,m)\mapsto g(x,m)$ seen at \cref{subsec:prop} leads to:
$$\psh{X}{t_0}-\psh{g(X,t_0)\cdot X}{t_0}\leq 0 \mbox{ almost surely.}$$
So the expected value of a non-positive random variable is null. Then
\begin{align*}
 \psh{X}{t_0}-\psh{g(X,t_0)\cdot X}{t_0}&=0 \mbox{ almost surely}
\psh{X}{t_0}&=\psh{g(X,t_0)\cdot X}{t_0} \mbox{ almost surely.}
\end{align*}
Then $g=e_G$ maximizes the dot product almost surely. Therefore (as we know that $g(X,t_0)$ is unique almost surely, since $t_0$ is regular):
\begin{equation*}
g(X,t_0)=e_G \mbox{ almost surely,}
\end{equation*} 
which is a contradiction with Equation~\eqref{cone}. \qquad

\subsection{Proof of \cref{alla}: upper bound of the consistency bias}
\label{allaproof}
In order to show this Theorem, we use the following lemma:
\begin{lemma}
\label{lemmestephanie1}
We write $X=t_0+\epsilon$ where $\E(\epsilon)=0$ and we make the assumption that the noise $\epsilon$ is a subgaussian random variable. This means that it exists $c>0$ such that:
\begin{equation}
\forall m\in M=\R^n,\: \E(\exp(\psh{\epsilon}{m}))\leq c \exp\left(\frac{s^2 \|m\|^2}{2}\right).
\label{subgaussien}
\end{equation}
If for $m\in M$ we have:
\begin{equation}
\tilde \rho:=d_Q([m],[t_0])\geq s\sqrt{2\log (c|G|)},
\label{rhohyp}
\end{equation}
then we have:
\begin{equation}
\tilde \rho^2-\tilde \rho s\sqrt{8\log(c|G|)}\leq F(m)-\E(\|\epsilon\|^2).
\label{rhocontrol}
\end{equation}
\end{lemma}

\begin{proof} (of \cref{lemmestephanie1})
First we expand the right member of the inequality~\eqref{rhocontrol}:
\begin{equation*}
\E(\|\epsilon\|^2)-F(m)= \E \left(\underset{g\in G}{\max}( \|X-t_0\|^2-\|X-gm\|^2)\right)
\end{equation*}
We use the formula $ \|A\|^2 -\|A+B\|^2=-2\psh AB-\|B\|^2$ with $A=X-t_0$ and $B=t_0-gm$:
\begin{equation}
\E(\|\epsilon\|^2)-F(m)= \E\left[\underset{g\in G}{\max }\left(-2\psh{X-t_0}{t_0-gm}- \|t_0-gm\|^2\right)\right]
=\E(\underset{g\in G}\max \: \eta_g),\label{max}
\end{equation}
with $\eta_g=-\|t_0-gm\|^2+2\psh{\epsilon}{gm-t_0} $. Our goal is to find a lower bound of $F(m)-\E(\|\epsilon\|^2)$, that is why we search an upper bound of $\E(\underset{g\in G}{\max} \eta_g)$ with the Jensen's inequality. We take $x>0$ and we get by using the assumption~\eqref{subgaussien}:
\begin{align}
\exp(x\E(\underset{g\in G}{\max}\: \eta_g))&\leq \E(\exp(\underset{g\in G}{\max}\: x\eta_g))  \leq \E\left( \underset{g\in G}{\sum} \exp({x\eta_g})\right) \nonumber\\
&\leq  \underset{g}{\sum} \exp(-x\|t_0-gm\|^2)\E(\exp(\psh{\epsilon}{2x(gm-t_0)}) \nonumber \\
&\leq c \underset{g}{\sum} \exp(-x\|t_0-gm\|^2)\exp(2s^2 x^2\|gm-t_0\|^2) \nonumber\\
&\leq c \underset{g}{\sum} \exp(\|gm-t_0\|^2(-x+2x^2s^2)) \label{sommeg}
\end{align}
Now if $(-x+2t^2x^2)< 0$, we can take an upper bound of the sum sign in~\eqref{sommeg} by taking the smallest value in the sum sign, which is reached when $g$ minimizes $\|g\cdot m-t_0\|$ multiplied by the number of elements summed. Moreover $(-x+2x^2s) < 0 \Longleftrightarrow 0< x<\frac1{2s^2}$. Then we have:
\begin{equation*}
\exp(x\E(\underset{g\in G}{\max} \:\eta_g))\leq  c |G|\exp(\tilde{\rho}^2(-x+2x^2 s^2)) \text{ as soon as } 0< x<\frac1{2s^2}.
\end{equation*}
Then by taking the log:
\begin{equation}
\E(\underset{g\in G}{\max} \eta_g)\leq  \frac{\log c|G|}{x}+(2xs^2-1)\tilde{\rho}^2.
\label{toptimal}
\end{equation}
Now we find the $x$ which optimizes inequality~\eqref{toptimal}. By differentiation, the right member of inequality~\eqref{toptimal} is minimal for $x_\star =\sqrt{\log c|G|/2}/(s\tilde \rho)$ which is a valid choice because $x_\star \in (0,\frac1{2s^2})$ by using the assumption~\eqref{rhohyp}. With the equations~\eqref{max} and~\eqref{toptimal} and $x_\star$ we get the result.\qquad
\end{proof}

\begin{proof}(of \cref{alla}) We take $m_\star\in \mbox{argmin }F$, $\tilde \rho=d_Q([m_\star], [t_0])$, and $\epsilon=X-t_0$. We have: $F(m_\star)\leq F(t_0)\leq \E(\|\epsilon\|^2)$ then $F(m_\star)-\E(\|\epsilon\|^2)\leq 0$. If $\tilde \rho> s \sqrt{2\log(|G|})$ then  we can apply \cref{lemmestephanie1} with $c=1$. Thus:
\begin{equation*}
\tilde \rho^2-\tilde \rho s \sqrt{8\log(|G|)}\leq 2F(m_\star)-\E(\|\epsilon\|^2)\leq 0,
\end{equation*}
which yields to $\tilde \rho\leq s \sqrt{8\log (|G|)}$. If $\tilde \rho\leq s \sqrt{2\log(|G|})$, we have nothing to prove. \qquad
\end{proof}

Note that the proof of this upper bound does not use the fact that the action is isometric, therefore this upper bound is true for every finite group action.

\subsection{Proof of \cref{2p}: inconsistency in $\R^2$ for the action of translation}
\label{p2p}

\begin{proof}
We suppose that $\E(X)\in HP_A\cup L$. In this setting we call $\tau(x,m)$ one of element of the group $G=\T$ which minimises $\|\tau\cdot x-m\|$ see~\eqref{deftau} instead of $g(x,m)$. The variance in the quotient space at the point $m$ is:
\begin{equation*}
F(m)=\E\left(\underset{\tau \in \Z/2\Z}{\min} \|\tau \cdot X-m\|^2\right)=\E(\| \tau(X,m)\cdot X-m\|^2).
\end{equation*}
As we want to minimize $F$ and $F(1\cdot m)=F(m)$, we can suppose that $m\in HP_A\cup L$. We can completely write what take $\tau(x,m)$ for $x\in M$:
\begin{itemize}
\item If $x\in HP_A\cup L$ we can set $\tau(x,m)=0$ (because in this case $x,\: m$ are on the same half plane delimited by $L$ the perpendicular bisector of $m$ and $-m$).
\item If $x\in HP_B$ then we can set $\tau(x,m)=1$ (because in this case $x,\: m$ are not on the same half plane delimited by $L$ the perpendicular bisector of $m$ and $-m$).
\end{itemize}
This allows use to write the variance at the point $m\in HP_A$:
\begin{equation*}
F(m)=\left(\E\left(\|X-m\|^2\fc_{\{X\in HP_A \cup L\}}\right)+\E\left(\|1\cdot X-m\|^2\fc_{\{X\in HP_B\}}\right)\right)
\end{equation*}
Then we define the random variable $Z$ by: $Z=X\fc_{X\in HP_A\cup L}+1\cdot X\fc_{X\in HP_B}$, such that for $m\in HP_A$ we have: $F(m)= \E(\|Z-m\|^2)$ and $F(m)=F(1\cdot m)$. Thus if $m_\star$ is a global minimiser of $F$, then $m_\star=\E(Z)$ or $m_\star=1\cdot \E(Z)$.
So the Fréchet mean of $[X]$ is $[\E(Z)]$. Here instead of using \cref{theo}, we can work explicitly: Indeed there is no inconsistency if and only if $\E(Z)=\E(X)$, ($\E(Z)=1\cdot \E(X)$ would be another possibility, but by assumption $\E(Z),\: \E(X)\in HP_A$), by writing $X=X\fc_{X\in HP_A}+X\fc_{X\in HP_B\cup L}$, we have:
\begin{align*}
\E(Z)=\E(X)&\Longleftrightarrow \E(1\cdot X \fc_{X\in HP_B\cup L})=\E(X\fc_{X\in HP_B\cup L})\\
&\Longleftrightarrow  1\cdot \E( X \fc_{X\in HP_B\cup L})=\E(X\fc_{X\in HP_B\cup L})\\
&\Longleftrightarrow \E(X\fc_{X\in HP_B\cup L})\in L\\
&\Longleftrightarrow \P(X\in HP_B)=0,
\end{align*}
Therefore there is an inconsistency if and only if $\P(X\in HP_B)>0$ (we remind that we made the assumption that $\E(X)\in HP_A\cup L$). If $\E(X)$ is regular (i.e. $\E(X)\notin L$), then there is an inconsistency if and only if $X$ takes values in $HP_B$, (this is exactly the condition of \cref{theo}, but in this particular case, this is a necessarily and sufficient condition). This proves point~1. Now we make the assumption that $X$ follows a Gaussian noise in order compute $\E(Z)$ (note that we could take another noise, as long as we are able to compute $\E(Z)$). For that we convert to polar coordinates: $(u,v)^T=\E(X)+(r\cos \theta,r\sin \theta)^T$ where $r>0$ et $\theta\in [0,2\pi]$. We also define: $d=\mbox{dist}(\E(X),L)$, $\E(X)$ is a regular point if and only if $d>0$. We still suppose that $\E(X)=(\alpha,\beta)^T\in HP_A\cup L$. First we parametrise in function of $(r,\theta)$ the points which are in $HP_B$:
 \begin{align*}
   v<u & \Longleftrightarrow  \beta+r\sin \theta< \alpha+r\cos \theta \Longleftrightarrow  \frac{\beta-\alpha}{r}< \sqrt 2 \cos(\theta+\frac \pi4)\\
   &\Longleftrightarrow \frac dr< \cos(\theta+\frac \pi 4) \\
   &\Longleftrightarrow   \theta \in \left[-\frac\pi4 -\arccos(d/r), -\frac\pi 4+\arccos(d/r)\right] \mbox{ and } d<r
\end{align*}
Then we compute $\E(Z)$:
\begin{align*}
\E(Z)=&\E( X \fc_{X\in HP_A})+\E(1\cdot X\fc_{X\in HP_B})\\
\E(Z)=&\int_{0}^{d}\int_{0}^{2\pi} \begin{pmatrix} \alpha+r\cos \theta\\\beta+r\sin \theta\end{pmatrix}\frac{\exp\left(-\frac{r^2}{2s^2}\right)}{2\pi s^2}rd\theta dr\\ 
 &+\int_{d}^{+\infini}\int_{\arccos(\frac dr)-\frac \pi4}^{2\pi-\frac \pi4-\arccos (\frac dr)}
  \begin{pmatrix} \alpha+r\cos \theta\\\beta+r\sin \theta\end{pmatrix}\frac{\exp\left(-\frac{r^2}{2s^2}\right)}{2\pi s^2}rdrd\theta\\
 & +\int_{d}^{+\infini}  \int_{-\frac\pi4 -\arccos\left( \frac dr\right)}^{-\frac\pi 4+\arccos\left( \frac dr\right)} 
 \begin{pmatrix}\beta+r\sin \theta\\ \alpha+r\cos \theta\end{pmatrix}\frac{\exp\left(-\frac{r^2}{2s^2}\right)}{2\pi s^2}rdrd\theta\\
 =& E(X) + \int_d^{+\infini} \frac{r^2\exp(-\frac{r^2}{2s^2})}{\pi s^2} \sqrt 2g\left(\frac dr\right)dr\times( -1 , 1)^T,
\end{align*}
We compute $\tilde \rho=d_Q([\E(X)],[\E(Z)])$ where $d_Q$ is the distance in the quotient space defined in~\eqref{quotientdistance}. As we know that $\E(X),\: \E(Z)$ are in the same half-plane delimited by $L$, we have: $\tilde \rho=d_Q([\E(Z)],[\E(X)])=\|\E(Z)-\E(X)\|$.  This proves \cref{item2}, note that \cref{itema,itemb,itemc} are the direct consequence of \cref{item2} and basic analysis.\qquad
\end{proof}

\section{Proof of \cref{lemmaf}: differentiation of the variance in the top space}
\label{sec:plemmaf}

\begin{proof}
By triangle inequality it is easy to show that $E$ is finite and continuous everywhere. Moreover, it is a well known fact that $x\mapsto d_M(x, z)^2$ is differentiable at any $m\in M\setminus C(z)$ (i.e. $z\notin C(m)$) with derivative $-2\log_m (z)$. Now since:
\begin{align*}
|d_M(x, z)^2 -d_M(y, z)^2| &= |d_M(x, z)-d_M(y, z)\|d_M(x, z) + d_M(y, z)| \\
&\leq  d_M(x, y)(2d_M(x, z) + d_M(y, x)),
\end{align*}
we get in a local chart $\phi:U\to V\subset \R^n$ at $t = \phi(m)$ we have locally around $t$ that:
\begin{equation*}
h \mapsto d_M(\phi^{-1} (t), \phi^{-1}(t + h)),
\end{equation*} 
is smooth and $|d_M(\phi^{-1}(t), \phi^{-1}(t + h))| \leq C|h|
$ for a $C > 0$. Hence for sufficiently small~$h$, $|d_M(\phi^{-1} (t), z)^2-d_M(\phi^{-1}(t + h), z)^2 | \leq C|h|(2d_M(m, z) + 1)$. We get the result from dominated convergence Lebesgue theorem with $\E(d_M(m, X)) \leq  \E(d_M(m, X)^2 +1)<+\infini$.\qquad
\end{proof}

\bibliographystyle{alpha} \bibliography{bi}

\end{document}